\documentclass[a4paper, 11pt]{amsart} 

\usepackage{amsmath, amsthm, amssymb}
\usepackage{enumitem}
\usepackage{xcolor}
\usepackage{hyperref}
\hypersetup{
    colorlinks,
    linkcolor={red!50!black},
    citecolor={blue!50!black},
    urlcolor={blue!80!black}
}
\usepackage{geometry} 
\usepackage{tikz-cd}
\usepackage{stackengine}

\usepackage{overpic}
\usepackage[ansinew]{inputenc}
\usepackage{graphicx}
\usepackage[rightcaption]{sidecap}
\usepackage{caption}
\usepackage{subcaption}

\numberwithin{equation}{section}

\newtheorem{theorem}{Theorem}[section]
\newtheorem{corollary}[theorem]{Corollary}
\newtheorem{lemma}[theorem]{Lemma}
\newtheorem{proposition}[theorem]{Proposition}

\newtheorem{claim}[theorem]{Claim}

\newtheorem*{question}{Question}

\theoremstyle{definition}
\newtheorem{definition}[theorem]{Definition}
\newtheorem{remark}[theorem]{Remark}

\makeatletter
\newcommand{\setword}[2]{%
  \phantomsection
  #1\def\@currentlabel{\unexpanded{#1}}\label{#2}%
}
\makeatother

\newcommand{\cL}{\mbox{${\mathcal L}$}}
\newcommand{\cC}{\mbox{${\mathcal C}$}}
\newcommand{\cH}{\mbox{${\mathcal H}$}}

\newcommand{\wcF}{\mbox{${\widetilde {\mathcal F}}$}}
\newcommand{\cG}{\mbox{${\mathcal G}$}}
\newcommand{\cD}{\mbox{${\mathcal D}$}}

\newcommand{\cP}{\mbox{${\mathcal P}$}}
\newcommand{\wcG}{\mbox{$\widetilde{\mathcal G}$}}

\newcommand{\cF}{\mbox{${\mathcal F}$}}

\newcommand{\cE}{\mbox{${\mathcal E}$}}
\newcommand{\mt}{\mbox{${\widetilde M}$}}

\newcommand{\rrrr}{\mbox{${\mathbb R}$}}
\newcommand{\zzzz}{\mbox{${\mathbb Z}$}}
\newcommand{\eps}{\mbox{${\epsilon}$}}
\newcommand{\curve}{\zeta}

\title[Transverse $\rrrr$-covered foliations]{On transverse $\rrrr$-covered
minimal foliations}
\author[T. Barbot]{Thierry Barbot}
\address{Avignon Universit\'e, LMA, Campus Jean-Henri Fabre,  
301, Rue Baruch de Spinoza, F-84 916 Avignon Cedex 9}
\email{thierry.barbot@univ-avignon.fr}
\author[S.R. Fenley]{Sergio R.\ Fenley} 
\address{Florida State University, Tallahassee, FL 32306, USA}
\email{sfenley@fsu.edu}
\author[R. Potrie]{Rafael Potrie}
\address{Centro de Matem\'atica, Universidad de la Rep\'ublica, Uruguay \& IRL-IFUMI (CNRS) }
\email{rpotrie@cmat.edu.uy}
\urladdr{http://www.cmat.edu.uy/~rpotrie/}

\thanks{S. F. was partially supported by 
National Science Foundation
grant DMS-2054909. R.P. was partially supported by CSIC}

\begin{document}
 
 \begin{abstract}
We study minimal   transverse foliations which are $\rrrr$-covered.
If in addition the dimension of the ambient manifold
is 3, and the foliations are Anosov foliations we give necessary and sufficient conditions for the intersected foliation to be the orbit foliation of an Anosov flow. 

\bigskip
\noindent{\bf Keywords:} Transverse foliations, 3-manifolds, Anosov flows,
group actions.

\medskip
\noindent {\bf Mathematics Subject Classification 2020: } 
\ Primary: 57R30, 37E10, 37D20, 37C15, 37C86; 
\ Secondary: 53C12,  37D05, 37D30, 57K30.
\end{abstract}

\maketitle


\section{Introduction}

Codimension one foliations on manifolds have proven to be an important structure which allows one to gather relevant information about the topology of the manifold. This is specially true under some assumptions in the foliation, for instance, tautness. We refer the reader to \cite{Ca-Co} for a general introduction to foliations. 

Existence of pairs or tuples of pairwise transverse foliations can provide stronger topological information. This has been extensively studied, see \cite{Ma-Ts} and references therein. Such structures occur fairly often, for example, given any Anosov flow in a manifold, it provides a pair of transverse foliations intersecting in the flow foliation. In addition, in dimension 3, partially hyperbolic diffeomorphisms, under some orientability assumptions, give rise to transverse branching foliations that can be blown up to a transverse pair of foliations. We refer the reader to \cite{BFP} for more discussion. Recently, the study of bi-contact structures on 3-manifolds has seen a renewed interest, and pairs of transverse foliations also allow one to produce such bi-contact structures (see e.g. \cite{Bowden,Ka-Ro,Ho,Massoni} and references therein). 

Observe that in most of these situations, the foliations are only continuous but usually tangent to continuous distributions, and this is the regularity we shall assume in this article. In particular this simplifies the parametrization of leaves of the intersected foliation (see \S~\ref{sub.fol} for a precise definition of the regularity we will assume). 

The goal of this article is to study transverse foliations
and how it relates to dynamical systems, particularly
in dimension $3$. The following is our first result: 


\begin{theorem}\label{thm.one}
Let $\cF_1, \cF_2$ be a pair of minimal, $\rrrr$-covered foliations
in  a manifold  $M$ which are  transverse to each other. 
Suppose that at  least one of them is transversely orientable.
Let $\wcF_i$  be the respective
lifts  to  the universal cover. Then  either  every leaf  
of $\wcF_1$ intersects every leaf of $\wcF_2$ or  the actions
of $\pi_1(M)$ on the leaf   spaces of  $\wcF_1$  and $\wcF_2$
are  conjugate  to each other.
In the second case it follows that both
foliations are transversely orientable.
\end{theorem}

Here $\cF_i$ being $\rrrr$-covered means that the 
leaf space of $\wcF_i$ is homeomorphic to the reals $\rrrr$.
Some of the concepts appearing in the statement,
in addition to the precise regularity we will consider
on foliations  will be specified in \S~\ref{s.background}. Theorem \ref{thm.one} is proved in \S~\ref{sec.minimal} where more general statements are presented. The notion of having the actions of $\pi_1(M)$ on the leaf spaces of the foliations being conjugate for minimal foliations is equivalent to the notion of the foliations being \emph{uniformly equivalent} introduced in \cite{Th5},  and will be discussed together with the notion of isotopic/homotopic foliations, in \S~\ref{s.background}. 

Given a pair of transverse foliations $(\cF_1, \cF_2)$ we denote by  $\cG = \cF_1 \cap \cF_2$ the intersection foliation. When $M$ has dimension $3$ it is a one-dimensional foliation. 

In \cite{Ma-Ts}, S. Matsumoto and T. Tsuboi exhibited examples of a pair of transverse foliations
on the unit tangent bundle of a closed surface $S$ of genus $\geq 2,$ each foliation being isotopic to the weak  stable (respectively unstable) foliation 
of the geodesic flow of $S$, but where the intersection foliation $\cG$  is not homeomorphic to the foliation by orbits of the geodesic flow.
In particular this implies that the pair of transverse foliations is not isotopic
to the pair defined by the weak  foliations of the geodesic flow. It is proved in \cite{FP2}, that one of the foliations then contains a {\em Reeb annulus}. A {\em Reeb annulus} is a compact  annulus  $A$ in a leaf of 
$\cF_1$ or $\cF_2$ so that it is saturated by $\cG$, and 
given any lift to the universal  cover $\mt$, the  boundary
leaves are not separated in the leaf space
of $\wcG$ in the lift of the annulus (see figure~\ref{fig.reeb}).
A \textit{Reeb surface} is a Reeb annulus or a Mobius band in a leaf saturated by $\cG$ which is the quotient of a Reeb annulus by an involution.

\begin{figure}[ht]
\begin{center}
\includegraphics[scale=0.60]{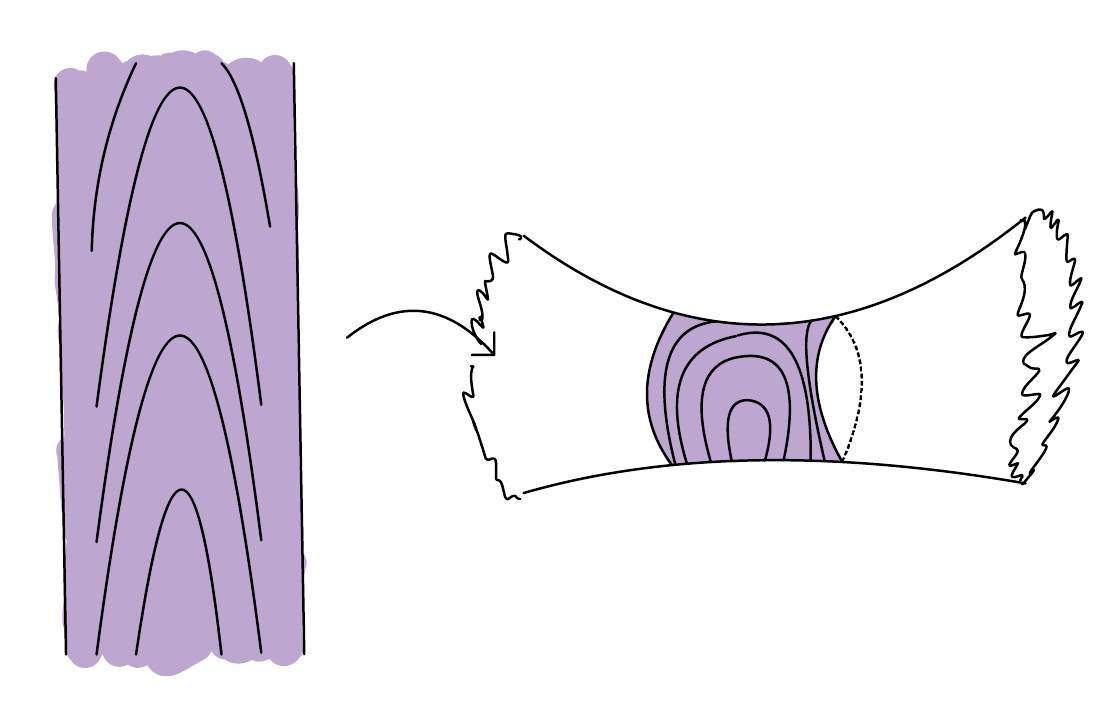}
\begin{picture}(0,0)
\end{picture}
\end{center}
\vspace{-0.5cm}
\caption{{\small A Reeb annulus in the universal cover and its projection to an annular leaf.}}\label{fig.reeb}
\end{figure}

 Let us recall the following result of  \cite{FP2}, that we will reprove here with
another proof: 

\begin{theorem} \label{thm.two}
Let $\cF_i$  be  minimal  foliations 
which are transverse to each  other in $M = T^1 S$, where
$S$ is a closed orientable,  hyperbolic surface.
Then either  there is 
a Reeb surface in a leaf  of $\cF_1$ or $\cF_2$; or
the intersection foliation is homeomorphic to the flow foliation
of the geodesic flow of $S$.
\end{theorem}

In this article we generalize this result for every pair of codimension one foliations in three manifolds
satisfying the following 
assumptions: 
\begin{enumerate}
\item the  foliations are minimal (meaning that every leaf is dense in the ambient manifold), 
\item the foliations
are what is  called  $\rrrr$-covered,  that is, the  lifts
to the universal  cover $\mt$  have leaf space homeomorphic
to the real line, 
\item the foliations are Anosov foliations, and 
\item both foliations are uniformly equivalent to each other. 
\end{enumerate}

Assumption $(1)$ actually follows from $(2)$ and $(3)$ but we single it out for its importance. 
Note that in Theorem \ref{thm.two} there is no need to assume that the foliations are Anosov since a result of Matsumoto \cite{Mat} implies that any minimal foliation in $T^1 S$ is homeomorphic to an Anosov foliation. The fact that the foliations are uniformly equivalent is not needed for Theorem \ref{thm.two} due to specific properties of Seifert manifolds (see Theorem \ref{thm.four} below).



We obtain a result on more general manifolds:

\begin{theorem}\label{thm.three}
Let $\cF_i$  be  $\rrrr$-covered, skew  Anosov foliations in a closed 3-manifold $M$
which are transverse to each other.
Suppose that the
actions of $\pi_1(M)$ on the  leaf spaces of $\cF_1$
and $\cF_2$ are  conjugate to each other. 
Then either $\cF_1$ and $\cF_2$ both contain a Reeb surface of $\cG$; or
the intersection foliation is homeomorphic to the flow foliation
of an Anosov flow. 
\end{theorem}

Roughly by skew Anosov flow we mean an Anosov flow in dimension
$3$ which is $\rrrr$-covered, but not orbitally equivalent to
a suspension Anosov flow. See more details in \S~\ref{s.background}. We point out here that besides generalizing Theorem \ref{thm.two}, the main point of this paper is to provide a completely different approach for the proof. This new approach involves the creation of a \emph{developing map} in the product of the leaf spaces (see Definition \ref{developingmap}). 
In fact there is an associated developing map for any pair of Reebless
transverse foliations. However in general the leaf spaces are only 
simply connected possibly non Hausdorff $1$-manifolds, and the
analysis is more complicated. In our situation of $\rrrr$-covered foliations
the product of leaf spaces is the plane, and to a certain extent
this reduces part of the analysis from dimension $3$ to dimension $2$. Still one could pose the following question in general\footnote{After this paper was released, Ellis Buckminster and Sam Taylor pointed out that one can produce examples on which the image is not $\mathbb{R}^2$ by intersecting the weak stable foliation of a non-$\mathbb{R}$-covered Anosov flow with the tilt of itself. See \cite{BT}.}:

\begin{question}
Let $\cF_1$ and $\cF_2$ two transverse (not necessarily $\rrrr$-covered) minimal foliations of codimension one, and denote by $\cL_1 = \mt/_{\widetilde{\cF_1}}$ and $\cL_2 = \mt/_{\widetilde{\cF_2}}$ their leaf spaces in the universal cover. Let $\cD: \mt \to \cL_1 \times \cL_2$ be the map that sends each point $x \in \mt$ to the leaves it belong, that is, $\cD(x) = (\widetilde{\cF_1}(x), \widetilde{\cF_2}(x))$. Is it possible that the image of $\cD$ is not homeomorphic to $\rrrr^2$?
\end{question}

The map $\cD$ defined in the previous question is the developing map mentioned before. The main idea in the proof of Theorem \ref{thm.three} is to analyze this natural map $\cD$ and show that if it does not intersect any invariant graph in the product of
leaf spaces (which in our case of $\rrrr$-covered foliations
is the plane), then the flow of the intersected foliation must be Anosov.
On the other hand if the image of this natural map
intersects an invariant graph, then one can produce some embedded tori in $M$ whose existence allows us to produce Reeb surfaces. In contraposition, the approach in \cite{FP2} involves a more involved study of the intersected foliation inside each leaf of the pair of foliations using coarse geometric techniques, and requires some three dimensional arguments which are way more involved in order to produce the Reeb surfaces. We hope that these complementary approaches can in the future be combined in some ways to continue to extend these kind of results for more general pairs of transverse foliations. 

In particular we point out here that by the way the proof is structured, the proof of Theorem \ref{thm.three} in the case where $M$ is a hyperbolic 3-manifold is much simpler (see Remark \ref{rem-hyperbolic}). The general case involves proving a technical result that is valid for general transverse foliations in 3-manifolds (see \S~\ref{sec-nonseparatedleaves}) which may be of independent interest. 

Observe that Theorem \ref{thm.one} means that the hypothesis of Theorem \ref{thm.three} are satisfied if there is a leaf of $\wcF_1$ which
does not  intersect every leaf of $\wcF_2$.

Theorem \ref{thm.two} is a consequence of the following
more general  result which is proved in the last section.
The generalization consists of enlarging the  class
of manifolds from $T^1 S$ to circle bundles over 
hyperbolic surfaces.

\begin{theorem} \label{thm.four}
Suppose  that  $\cF_1, \cF_2$ are transverse foliations in
a Seifert manifold. 
Suppose that $\cF_1, \cF_2$ are Anosov  foliations
and that $\cF_1$ does not contain Reeb surfaces
of  the intersection foliation $\cG$. Then $\cG$ is 
homeomorphic to the flow foliation 
of an Anosov flow.
\end{theorem}

Finally we make the following remark:
If the foliations $\cF_i$ are produced by approximating 
the  center stable and center unstable branching foliations of a partially
hyperbolic diffeomorphism, then there cannot be any
Reeb surfaces in leaves of $\cF_i$ \cite{FP2}. Hence the results
above  are relevant in the classification of partially hyperbolic diffeomorphisms\footnote{After this paper was submitted, some parts of this paper were used in the recent article \cite{FP4} to complete the classification of partially hyperbolic diffeomorphisms admitting only branching foliations with Gromov hyperbolic leaves.}.   

In \S~\ref{s.background} we provide necessary background for the rest of the paper as well as discuss some of the concepts appearing in the main statements. In \S~\ref{sec.minimal} we prove Theorem \ref{thm.one}. Section \ref{s.translationfoliations} is devoted to \textit{translation foliations} (see definition \ref{def:translation}) giving extensions of some results of \cite{Ma-Ts} as well as removing some smoothness assumptions. This section also discusses the case of transverse product Anosov foliations proving Proposition \ref{productproduct}. 
Some of the results in \S~\ref{s.translationfoliations} are valid in all dimensions as in \S~\ref{sec.minimal}. The same happens in \S~\ref{sec:invariant} which deals with group invariant monotone graphs in the phase space. Section \ref{sec-nonseparatedleaves} starts assuming that the ambient manifold is 3-dimensional, though in this section we prove a result that does not assume that the codimension one foliations are $\rrrr$-covered which we feel may be useful in other contexts. 
The goal of sections \ref{sec-groupinv} through \ref{sec.intersect}
is to prove Theorem \ref{thm.three} under
orientability conditions. 
In section \S~\ref{s.lastone} we deal with the general case
of Theorem \ref{thm.three}.
 Finally, in \S~\ref{s.circlebundle} we specialize to Seifert manifolds and prove Theorem \ref{thm.four}.  

\vskip .1in
\noindent
{\em {Acknowledgement}} $-$ We thank the referee of this article for invaluable
comments that greatly helped the presentation of this article.

\section{Background}\label{s.background}

\subsection{Foliations and transversality}\label{sub.fol}

A \emph{foliation} $\cF$ of a closed manifold $M$ is a partition of $M$ by immersed submanifolds (called \emph{leaves}) of the same dimension which locally pile up as a product. The foliation is said to be \emph{codimension one} if leaves are of dimension one less than the ambient manifold. We say that two foliations $\cF$ and $\cF'$ are \emph{homeomorphic} if there is a homeomorphism of $M$ mapping leaves of $\cF$ to leaves of $\cF'$. When the foliations are one-dimensional, we sometimes write \emph{topologically equivalent} instead of homeomorphic since it is reminiscent to the equivalence of flows in a manifold. We will discuss other notions of equivalence in \S~\ref{sub.homotopic}. 

The regularity for a pair of {\em transverse foliations} $\cF_1, \cF_2$
in this article is as follows: each $\cF_i$ is $C^{1,0}$, i.e. 
the leaves are $C^1$ immersed surfaces, the charts are assumed
only $C^0$, see further details in \cite{Ca-Co}. 
By transverse we mean that there is a combined foliated atlas
with local charts having coordinates $(x,y,z_1, \cdot, z_k)$, where
locally $\cF_1$ is given by $x = cte$ and $\cF_2$ is given by $y = cte$.
In addition we assume that the leaves of $\cG$ are $C^1$ (this kind of assumption is always met in the settings mentioned in the introduction).

We say that a foliation is \emph{minimal} if every leaf is dense in $M$. 

An important result by Haefliger on codimension one foliations provides conditions on which a foliation does not admit nulhomotopic transversals. When this is the case, leaves of $\cF$ lift to $\mt$, the universal cover of $M$, as hyperplanes and the \emph{leaf space}, which is the quotient $\mt/_{\wcF}$ is a one-dimensional, not necessarily Hausdorff, simply connected manifold. See \cite{Ca-Co} for a general introduction. We will mostly specialize to some cases that we shall now describe.

\subsection{$\rrrr$-covered foliations}\label{sub.defrrr}

As explained in the previous subsection, the leaf space $\cL$ of a foliation $\cF$ is the quotient of 
the lifted foliation  $\wcF$ of $\cF$ to the universal
cover $\mt$ by the relation  of being in the same leaf  of 
$\wcF$. We denote by $\mu: \mt \to \cL$ the projection map. We equip $\cL$ with the quotient topology. Observe that 
the fundamental group $\pi_1(M)$ acts naturally on $\cL.$ 

\begin{definition}
A foliation $\cF$ is said to be $\rrrr$-covered if its 
leaf space  $\cL$ is homeomorphic to $\rrrr$.
\end{definition}

A very easy result is the following:

\begin{lemma}
Suppose that $\cF$ is a foliation in a manifold
$M$. Suppose that $\cF$ is $\rrrr$-covered. Then $\cF$ is
codimension one, leaves
of $\wcF$ separate $\mt$, and  leaves of $\wcF$ are properly embedded.
\end{lemma}

\begin{proof}
Leaves of $\wcF$  disconnect $\mt$ 
since the real line minus a point is disconnected.
This  also  shows  that in a foliation box  $B$ of  $\wcF$  the
map between  the leaf  space of $\wcF|_B$ and the leaf space  of
$\wcF$ is injective.
It follows  that leaves of  $\wcF$  
are fibers of the topological submersion $\mu: \mt \to \cL \approx \rrrr$,
and are therefore properly embedded hypersurfaces  in  $\mt$.
\end{proof}

Once $\wcF$ is equipped with a transverse orientation, the leaf space $\cL$ admits a corresponding total order $<$. Every element of $\pi_1(M)$ either preserves this order or reverses it.

For every leaf $L$ of $\wcF$
(sometimes we will denote by $x$), we denote by $L^+$ (or $x^+$) the connected component of $\mt \setminus L$ (respectively $\mt \setminus x$) such that for every leaf $L'$ (resp. $x'$) of $\wcF$, contained in this connected component, and  intersecting a transversal that contains $L$ (or $x$) we have $\mu(L') > \mu(L)$ (or $\mu(x') > \mu(x)$). We denote by $L^-$ or $x^-$ the other connected component.

An $\rrrr$-covered foliation $\cF$ is said to be \emph{uniform} (see \cite[Definition 2.1]{Th5}) if in the universal cover any pair of leaves $L, L' \in \wcF$ verify that their Hausdorff distance is bounded. See \S~\ref{uniform} below.

\subsection{Topological equivalence versus homotopic equivalence}\label{sub.homotopic}

Let us recall that two foliations (of any codimension) $(M, \cF)$ and $(M, \cF')$ are {\em homeomorphic} or \emph{topologically equivalent} if there is a homeomorphism $h: M \to M$ mapping every leaf of $\cF$ to a leaf of $\cF'$. Note that in the case of oriented foliations of dimension $1$, a topological equivalence is also called {\em orbital equivalence.}

We will say that two foliations $(M,\cF)$ and $(M,\cF')$ are {\em isotopic} if they are topologically equivalent and the topological equivalence $h: M \to M$ defined above is isotopic to the identity. We warn the reader that this is a different notion than the classical notion of isotopy of foliations which requires the tangent distributions to be homotopic through integrable distributions. 

There is a weaker version: $\cF$ and $\cF'$ are {\em homotopically equivalent} if there are continuous maps $h: M \to M$ and $h': M \to M$ such that $h$ maps every leaf of $\cF$ into a leaf of $\cF'$, and that $h'$ maps every leaf of $\cF'$ into a a leaf of $\cF$, and such that $h \circ h'$ and $h' \circ h$ are both homotopic to the identity through homotopies preserving the foliations (for more details, see for example section $5$ of \cite{Ma-Ts}). 

In particular, $h$ is transversely injective: it maps different leaves of $\cF$ on different leaves of $\cF'$. But a homotopical equivalence is not necessarily a topological equivalence -  the maps $h$ and $h'$ may fail to be injective along leaves.

However, homotopical equivalence is easier to detect: according to \cite{Hae}, if $\cF$ and $\cF'$ have the same holonomy groupo\"{\i}d, and if they are both {\em classifying spaces} of this holonomy groupo\"{\i}d, then they are homotopically equivalent.

It is therefore interesting to detect when a foliation $\cF$ is a classifying space of its holonomy groupo\"{\i}d: it means that the holonomy covering of any leaf $F$ is contractible, i.e. that the covering space associated to the holonomy morphism of $F$ is contractible as a topological space (see page $85$ of \cite{Hae}). 

\vskip .1in
\noindent
{\bf {The case of dimension $3$ and $\cF$ is $\rrrr$-covered}}

In the rest of this subsection, we only consider the specific case where $M$ has dimension $3$ and $\cF$ is $\rrrr$-covered. In this case, the foliation is a classifying space of its holonomy groupo\"{\i}d if and only if it has no spherical leaves and that the only element of $\pi_1(M)$ admitting an open set of fixed points on $\cL$ is the trivial element.

Using this, one can show that if $\cF_1$ and $\cF_2$ are $\rrrr$-covered foliations and the actions of $\pi_1(M)$ on the leaf  spaces of  $\wcF_1$  and $\wcF_2$
are  conjugate  to each other then, the foliations are homotopically equivalent. 

By conjugate to each other we mean the following:
if we denote by $\cL_i$ the leaf spaces of $\wcF_i$ ($i=1$ or $2$), then there is a homeomorphism 
$f: \cL_1 \to \cL_2$ such that for every $\gamma$ in $\pi_1(M)$ we have $f(\gamma.x) = \gamma.f(x).$ In other words, the automorphism $\varphi$ of $\pi_1(M)$
for which $f$ is $\varphi$-equivariant is trivial. It means that if $f$ is induced by some homeomorphism $F$ of $M$ mapping $\cF_1$ to $\cF_2$ then $F$ is homotopic to the identity.

We also note that in \cite{Th5}, still in dimension $3$, it is shown that  for minimal $\rrrr$-covered uniform foliations (see \S~\ref{uniform}), these notions coincide. In particular, in \cite{Th5} the notion of foliations $\cF_1$ and $\cF_2$ being \emph{uniformly equivalent} is presented. It means that for every leaf $L$ of $\wcF_1$ there is a leaf $E \in \wcF_2$ at bounded Hausdorff distance from $L$ and viceversa.  It is shown (see \cite[Proposition 2.2]{Th5}) that two minimal $\rrrr$-covered uniform foliations which are uniformly equivalent must be homeomorphic by a homeomorphism at bounded distance from the identity (in particular,  isotopic in the sense discussed above).

\subsection{Translation foliations}

\begin{definition}\label{def:translation}
A codimension one foliation $\cF$ is a \textit{translation foliation} if it is $\rrrr$-covered and the action of
the fundamental group on $\cL$ is topologically conjugated to an action by translations.
\end{definition}

In this section, we explain that an alternative definition of translation foliations is that they are the codimension one foliations that are transversely oriented and admitting a transverse Riemannian structure (equivalently, that they are the ones transversely modeled on the Lie group $\rrrr$). It is well known by classical results that such a foliation is $\rrrr$-covered (see \cite{Car}). However, the only proofs we know in the literature assumes differentiability of the foliation, hence we provide here a proof, whose arguments will moreover be used later.

\begin{lemma}\label{le:translationR}
Let $(M, \cF)$ be a codimension one foliation locally modeled on the Lie group $\rrrr.$ Then, it is a translation foliation.
\end{lemma}

\begin{proof}
The hypothesis mean that there is a continuous map $d: \mt \to \rrrr$ and a morphism $r: \pi_1(M) \to \rrrr$ such that
$d$ is constant along $\wcF$ and locally injective along any small transversals to $\wcF$, and such that, for all $\tilde p$ in $\mt$ and 
any $\gamma$ in $\pi_1(M)$ we have $d(\gamma \tilde{p}) = d(\tilde{p}) + r(\gamma).$ The Lemma will be proved if we show that $d$ induces a homeomorphism between the leaf space $\cL$ and the real line $\rrrr$. 

First, it follows from the existence of $d$ that a leaf of $\wcF$ cannot intersect a given transversal in two different points, since these points have the same image under $d$ and that the restriction of $d$ to such a transversal is injective (since any local homeomorphism from an interval into $\rrrr$ is injective). It follows that as a topological space, $\cL$ is a $1$-dimensional manifold. The point is that $\cL$ 
may be not Hausdorff.

We choose an auxiliary metric on $M.$ For every $p$ in $M$ consider a foliated chart $U_p$ homeomorphic to $F \times (-1, 1)$ such that for every $t$ in $(-1, 1)$ the set $F \times \{ t \}$ is a small disk in a leaf of $\cF$ and such that $p$ lies in $F \times (-1/4,1/4)$. By compactness of $M$, we can assume that there exists $\epsilon_- > 0$, $\epsilon_+ > 0$
such that for every $p$ the chart $U_p$ contains the ball of radius $\epsilon_-$ centered at $p$, and that $p$ is at distance at most $\epsilon_+$ from $F \times \{ 1 \}$ and from $F \times \{ -1 \}$ in $U_p$.

Moreover, we can assume that there is $C > 0$ such that for any lift $U_{\tilde p} \approx \widetilde{F} \times (-1, 1)$ in $\mt$ the value of $d$ on $\widetilde{F} \times \{1 \}$ (respectively  $\widetilde{F} \times \{-1 \}$) is bigger than $d(\tilde p) + C$ (respectively smaller than $d(\tilde p) - C$).

Then, for every $\tilde p$ in $\mt$ and every real number $\varepsilon$ of absolute value $\leq C$ consider every path in $\mt$ of length $\leq \epsilon_+$
transverse to $\wcF$ starting from $\tilde p$ and ending at a point $\tilde q$ such that $d(\tilde q) = d(\tilde p) + \varepsilon$. If $\epsilon_+$ is selected small enough, all the final points of all these curves are all in the same leaf of $\wcF$, that we call $F^{\varepsilon}(\tilde p)$. The map $\Tilde{p} \mapsto F^{\varepsilon}(\Tilde{p})$ is clearly locally constant along $\wcF$, hence constant on the leaf $\wcF(\tilde{p})$. 

Assume that there are leaves $F_1$ and $F_2$ of $\wcF$ that are not separated in $\cL$. 
Let $\Tilde{p}_1$ an element of $F_1$ and $\Tilde{p}_2$ an element of $F_2.$
Then, there are sequences of points $\Tilde{q}^1_n$ and $\Tilde{q}^2_n$ belonging to the same leaves $L_n$ of $\wcF$ (for each  $n$),
and converging to $\Tilde{p}_1$ and $\Tilde{p}_1$, respectively. For $n$ sufficiently big, $\Tilde{q}^i_n$ $(i=1,2)$ is very close to $\Tilde{p}_i$, and  therefore 
$L_n = F^{\varepsilon_n}(\Tilde{p}_i)$, for some $\varepsilon_n$.
This is a key point here: the $\varepsilon_n$ is the same for poth 
$\Tilde{p}_1$ and $\Tilde{p}_2$ because $d(\Tilde{p}_1) = d(\Tilde{p}_2)$,
as their leaves are non separated from each other.

Hence, $F_1 =  F^{-\varepsilon_n}(\Tilde{q}^1_n) =  F^{-\varepsilon_n}(\Tilde{q}^2_n) = F_2$ since $F^{-\varepsilon_n}$ is constant along $F_n.$ 

It follows that $\cL$ is a Hausdorff $1$-manifold, hence homeomorphic to the real line or the circle. Since $d$ induces a local homeomorphism $\Bar{d}$ from $\cL$ into $\rrrr$, the only possibility is $\cL \approx \rrrr.$ 
Then, $\Bar{d}$ is a homeomorphism from $\cL$ to an interval $I$ of $\rrrr.$
For every $x$ in $\cL$ the interval $[x - C, x+C]$ is contained in $I$ (since the image of $\Bar{d}$ contains 
$d(\tilde p) \pm C$ where $\tilde p$ is a point in the leaf $x$). Therefore $I = \rrrr.$ It finishes the proof of Lemma \ref{le:translationR}.
\end{proof}

Typical examples of translation foliations are the ones defined by closed $1$-forms. Actually, if the conjugacy to the action by translations is $C^1$, a translation foliation
is defined by a closed $1$-form: if $\eta: \widetilde{M} \to \cL \approx \rrrr$ is $C^1$, then the differential $d\eta$ defines a closed $1$-form on $\widetilde{M}$ 
(even exact), which is preserved by the action of $\pi_1(M)$ and descends to a closed $1$-form on $M$ that defines the foliation.

Tischler's argument (see \cite{Ca-Co}) shows that, when $C^1,$ every translation foliation can be approximated by foliations whose leaves are fibers of some fibration of $M$ over the circle.

\subsection{The phase space of a pair of transverse foliations}\label{ss.developing}

We now consider a pair $(\cF_1, \cF_2)$ of transverse $\rrrr$-covered foliations. We denote by $\mu_i: \mt \to \cL_i$ $(i=1,2)$ the projections to the leaf spaces.

We denote by  
$\cG:= \cF_1 \cap \cF_2$ the intersection foliation. Let $\wcG$ be its  lift to $\mt$. We denote by $\cL_{\cG}$ the leaf space of $\wcG.$ Be aware that $\cL_{\cG}$ may not be Hausdorff.

\begin{definition}[Developing map]\label{developingmap}
    The {\em phase space} is the product  $\cL_1 \times \cL_2$. The {\em projection map} (or \emph{developing map}) $\cD: \mt \to \cL_1 \times \cL_2$ is the map
    associating to every point in the universal covering $\mt$ the leaves of $\wcF_1$ and $\wcF_2$ it belongs to. The projection map $\cD$ induces a local homeomorphism $\cD_{\cG}$ between $\cL_{\cG}$ and $\cL_1 \times  \cL_2 \cong \rrrr^2.$ 
\end{definition}

Each $L$ in $\wcF_1$ is path connected and  therefore the
set  of  leaves of $\wcF_2$ which is intersected by
$L$ is a path connected set in $\cL_2$.

On the other hand, observe that the intersection between one leaf of $\wcF_1$ and one leaf of $\wcF_2$ might be not connected, implying that fibers of $\cD$ are not in general
leaves of the intersection foliation $\wcF_1 \cap \wcF_2$, but rather 
unions of such leaves. The following appears also in \cite[\S 2.2]{FP3} in the 3-dimensional case:

\begin{proposition}\label{pro.OK}
    Let $L$ be a leaf of $\wcF_1$. The following assertions are equivalent:
    \begin{enumerate}
        \item The restriction $\wcG_L$ of $\wcG$ to $L$ is $\rrrr$-covered,
        \item The intersection between $L$ and every leaf of $\wcF_2$ is connected, i.e. either empty, or a single leaf of $\wcG$.
    \end{enumerate}
\end{proposition}

\begin{proof}
    Assume that $\wcG_L$ is $\rrrr$-covered, and let $L'$ be a leaf of $\wcF_2$. Assume that $L \cap L'$ contains two different leaves $\theta_1$ and $\theta_2$ of $\wcG.$ Then, there is a curve $c$ in $L$, transverse to $\wcG_L$ containing $\theta_1$ and $\theta_2$. The restriction of $\mu_2: \mt \to \cL_2 \approx \rrrr$ to $c$ is locally injective, hence injective. It is a contradiction since $\mu_2(\theta_1) = L' = \mu_2(\theta_2)$. Hence item $(2)$ is true.

    Assume now item $(2)$. The same argument that  shows that leaves
of $\wcF_2$ are properly embedded in $\mt$ shows that leaves
of $\wcG_L$ are properly embedded  in $L$.
In particular this implies that the leaf space of $\wcG_L$ is
a $1$-manifold. Assume that $\wcG_L$ is not Hausdorff. Then, there are two leaves $\theta_1$ and $\theta_2$ of $\wcG_L$ that are not separated. They must have the same image under $\cD$, because  $\cD$ is continuous
and the  phase space  is Hausdorff.  Hence $\theta_1, \theta_2$
are contained in the same leaf of $\wcF_2$, contradiction.
Since every leaf of $\wcF_2$ disconnects $\mt$, every leaf of $\wcG_L$ disconnects $L.$ Hence, the leaf space $\wcG_L$ is not a circle, and it is homeomorphic to $\rrrr$. Item $(1)$ is true.  
\end{proof}

\begin{lemma}\label{le:lambdaL}
    Let $L$ be a leaf of $\wcF_1$. Let $\Lambda_L$ be the set of leaves $L'$ of $\wcF_2$ such that the intersection $L \cap L'$ is not connected. Then either $\Lambda_L$ is empty, or the interior of $\Lambda_L$ is not empty.
\end{lemma}

\begin{proof}
    Assume that $\Lambda_L$ is not empty. According to Proposition \ref{pro.OK}  the foliation $\wcG_L$ is not Hausdorff: there is a leaf $L'$ of $\wcF_2$ such that $L \cap L'$ contains two leaves $\theta_1$ and $\theta_2$ of $\wcG_L$ that are not separated. Let  $\theta_i^{\pm}$ be the connected
components of $L \setminus  \theta_i$. We choose them so that $\theta_1 \subset \theta_2^-$ and $\theta_2 \subset  \theta_1^-$. Let $c_1$, $c_2$ be two curves transverse to $\wcG_L$, the first crossing $\theta_1$, the second $\theta_2.$ Since $\theta_1, \theta_2$ are not separated  in $\wcG_L$, there are leaves of $\wcG_L$ intersecting $c_1$ and $c_2.$ More precisely, there is one connected component $c_1^-$ of $c_1 \setminus \theta_1$ and one connected component $c^-_2$ of $c_2 \setminus \theta_2$ such that $\mu_2(c_1^-) = \mu_2(c_2^-)$. We can choose the order $<_2$ (see the conventions at the end of section \ref{sub.defrrr}) of $\cL_2$ so that all elements $x$ in $\mu_2(c_1^-) = \mu_2(c_2^-)$ satisfy $x <_2 \mu_2(L').$ 
    Then $c_1^-$ is contained in $\theta_1^-$. 
%

    Let now $F$ be a leaf of $\wcF_2$, close to $L'$, and such that $\mu_2(F) > \mu_2(L').$ Since it is close to $L'$, $F$ must intersect both $c_1$ and $c_2.$ Since $\mu_2(F) > \mu_2(L')$, these intersections must be in $c^+_1$ and $c^+_2$, respectively. Hence $F \cap L$ contains a leaf of $\wcG$ in $\theta_1^+$, and another in $\theta_2^+.$ Since $\theta_1^+$ and $\theta_2^+$ are disjoint, these leaves are different: $F$ belongs to $\Lambda_L$. Actually, this is true for every leaf in the interval $(L', F]$ of $\cL_2:$ the Lemma is proved.    
\end{proof}

\begin{corollary}\label{cor.F1F2}
    Assume that $\cF_1$ is minimal. Assume that one leaf of $\wcF_1$ satisfies the condition stated in Proposition \ref{pro.OK}.
    Then the leaf space $\cL_{\cG}$ of $\wcG$ is Hausdorff, and the map $\cD_{\cG}: \cL_{\cG} \to \cL_1 \times \cL_2$ is a homeomorphism onto its image. 
    \end{corollary}

\begin{proof}
Let $\Lambda$ be the subset of $\cL_1$ comprising leaves of $\wcF_1$ that does not satisfy the condition stated in Proposition \ref{pro.OK}. By hypothesis, it is not the entire $\cL_1.$

Assume that $\Lambda$ is not empty. Let $F$ be an element of $\Lambda$. Let $L$ be a leaf of $\wcF_2$ intersecting $F$ at two different leaves of $\wcG.$

Then, $F$ belongs to $\Lambda_L$ (we use the notation in statement of Lemma \ref{le:lambdaL}, and we have switched the roles of $\cF_1$ and $\cF_2$). According to Lemma \ref{le:lambdaL}, the set $\Lambda_L$ has non-empty interior. Since $\Lambda_L$ is contained in $\Lambda$, the interior of $\Lambda$ is not empty.

But $\Lambda$ is $\pi_1(M)$-invariant, not empty, and $\wcF_1$ is minimal.

We conclude that $\Lambda$ is the entire $\cL_1$. 

But by hypothesis $\Lambda$ is not the entire $\cL_1.$ Therefore we have proved by a way of contradiction that $\Lambda$ is empty.

We deduce that the intersection between a leaf of $\wcF_1$ and a leaf of $\wcF_2$ is either empty, or a single leaf of $\wcG.$ It means that the continuous map $\cD_{\cG}: \cL_{\cG} \to \cL_1 \times \cL_2$ is injective. 
\end{proof}

Note that under the assumptions of the previous corollary, in dimension 3, we could apply the main results of \cite{FP3} to deduce that the intersected foliation is by quasigeodesics. Here, since we will look at specific foliations, we will give some more direct arguments (see \S~\ref{sec.nointersect}).

\begin{figure}[ht]
\begin{center}
\includegraphics[scale=0.60]{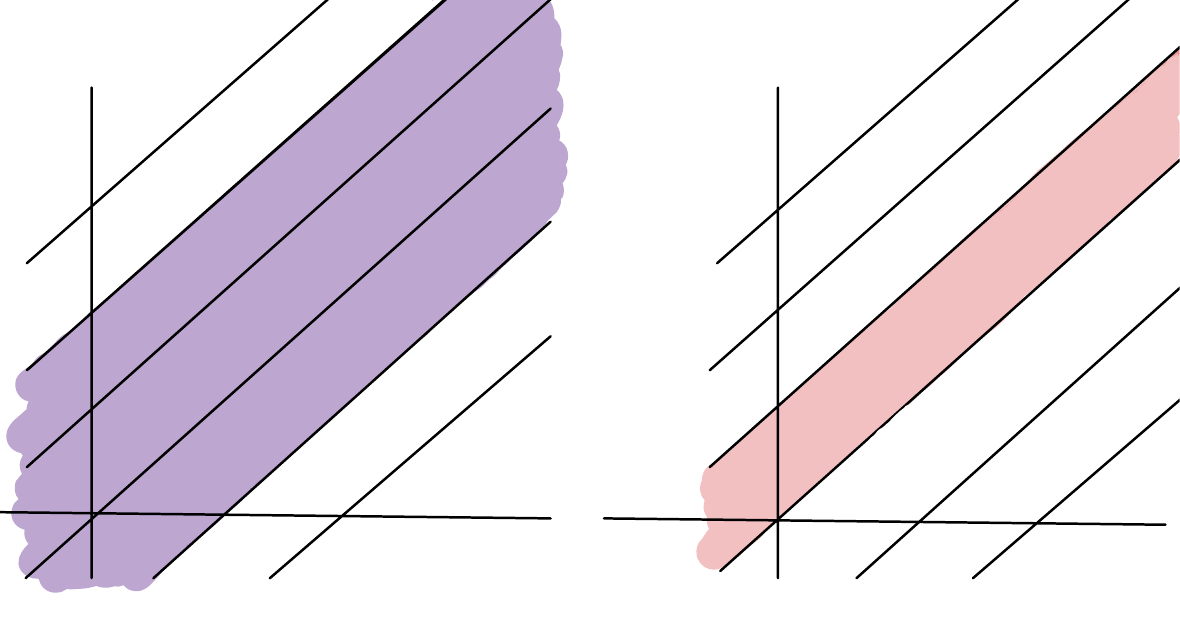}
\begin{picture}(0,0)
\end{picture}
\end{center}
\vspace{-0.5cm}
\caption{{\small Possibilities for the image of the developing map. When the developing map is not surjective, the image is bounded by graphs corresponding to conjugacies between the $\pi_1(M)$-actions on the leaf spaces (see Proposition \ref{pro.alphafini}). In the first case (see \S~\ref{sec.intersect}), we will find Reeb surfaces in the preimage of the invariant graph, while in the second (see \S~\ref{sec.nointersect}), assuming that there are no Reeb surfaces, we will show that the intersected foliation is homeomorphic to an Anosov foliation.}}\label{fig.developing}
\end{figure}

\subsection{Anosov foliations}\label{sub.defanosov}

Anosov flows provide interesting pairs of transverse foliations: their weak  stable and unstable foliations $\cF^s$, $\cF^u$.

In this section we recall the structure of \textit{$\rrrr$-covered Anosov flows}, i.e. Anosov flows $(M, \Phi)$ on a closed $3$-manifold such that one of the weak  foliations $\cF^s$, $\cF^u$ is $\rrrr$-covered. For details, we refer to \cite[\S 2.2]{Barb-Fe2}.

The first important fact is that if one of the weak  foliations is $\rrrr$-covered, then they are both $\rrrr$-covered \cite{Ba1,Fe1}. We denote by $\cL^s$, $\cL^u$ the leaf spaces of the lifted foliations $\wcF^s$, $\wcF^u$, which are both homeomorphic to $\rrrr.$ We only consider the case where $\cF^s$ and $\cF^u$ are transversely orientable, i.e. the case where the actions of $\pi_1(M)$ on $\cL^s$ and $\cL^u$ are orientation preserving.

\subsubsection{The product case}\label{sub.product}
The first case to consider is the {\em product case}, i.e. the case where every leaf of $\wcF^s$ intersects every leaf of $\wcF^u.$ It happens if and only if $\Phi$ is orbitally equivalent to the suspension of a linear hyperbolic automorphism $A$ of the $2$-dimensional torus. In this case, we denote by $(M(A), \Phi_A)$ the Anosov flow, where we take a constant first return time equal to $1$.

We have a split exact sequence:
$$0 \to H \to \pi_1(M) \to \zzzz \to 0$$
where $H$ is a normal subgroup isomorphic to $\zzzz^2$ corresponding to the fundamental group of the fibers.

The fundamental group is a semi-direct product, and the action of the quotient $\zzzz$ on $H$ is given by $A.$

Let $\lambda$, $\mu$ be the two eigenvalues of the monodromy $A$. We
describe here the transversely oriented case, in the
last section of the article we deal with the general
case. In the transversely orientable case,  $A$ has determinant $1$ and $\lambda$ and $\mu$ are both positive. We can take the convention that $\lambda  > 1$ and $ \mu  = 1/\lambda < 1$. The leaf spaces $\cL^s$ and $\cL^u$ are both identified with $\rrrr$, and for every element $\gamma$ of $\pi_1(M)$ there is an integer $k(\gamma)$ and real numbers $a(\gamma)$, $b(\gamma)$ such that the action of $\gamma$ on $\cL^s$ is $x \mapsto \lambda^{k(\gamma)}x + a(\gamma)$ and the action on $\cL^u$ is given by $y \mapsto \mu^{k(\gamma)}y + b(\gamma)$. Then, $k$ defines a morphism from $\pi_1(M)$ to $\zzzz$, with kernel $H$, and the set of all pairs $(a(\gamma), b(\gamma))$ forms a lattice in $\rrrr^2.$ Moreover, the integer $k(\gamma)$ is, up to sign, the projection of $\gamma$ in $\pi_1(M)/H\approx \zzzz.$

As a last remark, we point out that $\cF^s$ can be topologically conjugate to $\cF^u$, but not always: it depends on whether $A$ is conjugate to its inverse $A^{-1}$ in SL$(2, \zzzz)$ or not.

\subsubsection{The skew $\rrrr$-covered case }
The skew $\rrrr$-covered is the case where the Anosov flow $(M, \Phi)$ is $\rrrr$-covered but not product.

In the rest of this subsection we assume that the foliations
are transversely orientable.

The transverse foliations $\cF^s$ and $\cF^u$ are then isotopic one to the other. 
In particular, the actions of $\pi_1(M)$ on $\cL^s$ and $\cL^u$ are topologically conjugated. 

More precisely, there are two homeomorphisms  $\alpha_0, \beta_0$ from $\cL^s$ to $\cL^u$ that both commute with the action of $\pi_1(M)$
(the commutation property uses the transverse orientability
of the foliations). They are defined as follows: for every 
leaf $x$ of $\wcF^s$, the set of leaves of $\wcF^u$ intersecting $x$ is an interval $(\alpha_0(x), \beta_0(x) )$ in $\cL^u.$ In particular, 
the orbit space of the lift of $\Phi$ to $\mt$ is canonically identified with the domain $\Omega_0$ of $\cL^s \times \cL^u$
which is between  the  graphs of $\alpha_0$ and $\beta_0$. Observe that we can select the transverse orientations such that $\alpha_0$, $\beta_0$ are both increasing maps.

The maps $\tau_s = \alpha_0^{-1} \circ \beta_0$ and $\tau_u = \beta_0 \circ \alpha_0^{-1}$ are homeomorphisms of respectively $\cL^s$ and $\cL^u$ commuting with the $\pi_1(M)$ actions.
These two maps are  called the one step up  maps.

The following proposition shows that they are essentially the unique maps satisfying this property.

\begin{proposition}\label{pro:alphaLL}
The homeomorphisms of $\cL^s$ (respectively $\cL^u$) to itself commuting with the action of $\pi_1(M)$ are precisely the iterates $\tau^k_s$ (respectively $\tau_u^k$) with $k \in \zzzz$.

In particular, the only maps from $\cL^s$ to $\cL^u$ commuting with the actions of $\pi_1(M)$ are the maps of the form $\tau_u^k \circ \alpha_0 = \beta_0 \circ \tau_s^k.$   
\end{proposition}

\begin{proof}
See \cite[Corollary 2.8]{Barb-Fe2}.
\end{proof}

\begin{lemma}\label{le.fixedpoints}
    Let $\gamma$ be a non trivial  element of $\pi_1(M)$ admitting a fixed point in $\cL^s.$ Then, $\gamma$ admits infinitely many fixed points. Moreover, the fixed points of $\gamma$ can be indexed by $\zzzz$ so that, for every integer $i$:
    \begin{itemize}
        \item $x_i < x_{i+1}$,
        \item $x_{i+2}=\tau_x(x_i)$,
        \item $x_i$ is an attracting fixed point if $i$ is even, and a repelling one if $i$ is odd.
    \end{itemize}
\end{lemma}

In other words, the homeomorphism induced by $\gamma$ on the circle, quotient of $\cL^s$ by $\tau_s$, admits two fixed points, one attracting and the other repelling.

Finally, the following classical lemma will be useful 
(for a proof, see \cite[Lemma 3.1 and 3.2]{Ma-Ts}):

\begin{lemma}\label{le:periodicdiscrete}
    An element of the orbit space $\Omega_0$ is the lift of a periodic orbit if and only if its orbit under $\pi_1(M)$ is discrete in $\Omega_0.$ 
\end{lemma}


\subsection{Uniform foliations and slitherings} \label{uniform}

For this section we refer to the  seminal article of Thurston
\cite{Th5}.

\begin{definition} (uniform foliations) A codimension one foliation $\cF$ in
a manifold  is uniform, if given any two leaves $E, F$ of $\wcF$
in $\mt$, then the Hausdorff distance between $E$ and $F$
is finite.
\end{definition}

If $M$ has dimension $3$ and $\cF$ is uniform, and
Reebless, then $\cF$  is $\rrrr$-covered\footnote{A similar result is proved in \cite{Th5} but it implicitly assumes $\rrrr$-covered as an assumption. The goal in \cite{Th5} is to show that a uniform $\rrrr$-covered foliation is necessarily a blow up of a slithering.} \cite{FP1}.
The notion of uniform foliations is closely connected with
the notion  of  slitherings:

\begin{definition} (slithering) A slithering is a submersion $f: \mt \to  
{\mathbb{S}^1}$ which is equivariant: there is  a representation
$\rho: \pi_1(M) \to Homeo(\mathbb{S}^1)$ so that
$f(\gamma(p)) = \rho(\gamma)(f(p))$  for all  $\gamma$ in $\pi_1(M)$
and $p$  in $\mt$.
\end{definition}

A slithering defines a  codimension one foliation, first
in $\mt$ by taking connected components of inverse images
of points. This projects to a foliation in $M$ by the equivariance.
Thurston \cite{Th5} (see  also  \cite{Fe2,Cal1}) proved that
if $\cF$ is uniform, codimension one, and minimal then it
is defined by a slithering. 
A slithering defines a homeomorphism $\tau$ from the leaf
space $\cL$ of $\cF$: given a transverse orientation on $\cF$,
and a leaf $L$ of $\wcF$, let $\tau(F)$ be the next leaf 
of $\wcF$ on the  positive side of  $F$
which has the same value as $F$ under $f: \mt 
\to {\mathbb{S}^1}$.
This is called the step  one map on the level of the
leaf spaces.

If $\cF$ is a skew $\rrrr$-covered Anosov foliation,
then the maps $\tau_s, \tau_u$ (from $\cL_s$ to $\cL_s$
and $\cL_u$ to $\cL_u$)  defined previously
are the one step up slithering maps that we are alluding
to here.

An easy but important property will  be  used  here:

\vskip .1in
\noindent
{\bf {Property}} Let $\cF$ be defined by  a slithering
with a slithering map $\tau: \cL \to \cL$.
There is  $a_0 > 0$,  so that for any $L$ leaf of $\wcF$,
$p  \in L$ and
$q \in \tau(L)$, then  $d(p,q) > a_0$.
\vskip .1in

In other words the leaves $L, \tau(L)$ cannot get  arbitrarily
close to each other.

\subsection{On the orientability assumption}\label{ss.orientable} 

In the proof of Theorems \ref{thm.three} and \ref{thm.four} we
first assume that $\cF_1, \cF_2$ are transversely orientable
and $\cG$ is orientable. The general case of Theorem 
\ref{thm.three} is done in section \S~\ref{s.lastone}.
The proof of Theorem \ref{thm.one} is done
in \S~\ref{sec.minimal} under general conditions.


\section{Intersection of minimal foliations}\label{sec.minimal}

In this  section $\cF_1, \cF_2$ will be $\rrrr$-covered, transverse
to each other, and minimal, meaning that every leaf of $\cF_1$ (and of $\cF_2$ as well) is dense in $M$. We fix orders $<_1$, $<_2$ in each of $\cL_1, \cL_2$. One goal of this section is to prove Theorem \ref{thm.one}. 

\begin{definition}\label{def.orient}
We denote by $\pi_1(M)_0$ the subgroup of $\pi_1(M)$ comprising elements preserving the order $<_2.$
\end{definition}

Observe that $\pi_1(M)_0$ is a normal subgroup of $\pi_1(M)$ of index at most $2.$ Therefore, the actions of $\pi_1(M)_0$ on $\cL_1$ and $\cL_2$ are still minimal.

\begin{definition} (the maps  $\alpha, \beta, \alpha', \beta'$)\label{def.alphabeta}
For  each $L$ in $\cL_1$, define 
$\alpha(L)$ as follows:

$$\alpha(L) \ \ = \ \ {\rm inf} \{ U \in  \cL_2, \ \ U \cap L \not
= \emptyset \},$$

\noindent
if  the infimum exists, in  which case it is a leaf of $\cL_2$.
Otherwise define  $\alpha(L)$ to be $-\infty$.
Analogously, to define $\beta(L)$ use the supremum instead
of the infimum, and $\beta(L)$ has  the possible value $+\infty$.
Similarly, for each E in $\cL_2$,  define 

$$\alpha'(E) \ \ = \ \ {\rm inf} \{ V \in \cL_1, \ \ V \cap E
\not = \emptyset \}$$

\noindent
with the same convention as for $\alpha$. Analogously define
$\beta'$.
\end{definition}

Notice $\alpha: \cL_1 \to \cL_2 \cup \{ -\infty \}$,
$\beta: \cL_1 \to \cL_2 \cup \{ +\infty \}$,
$\alpha': \cL_2 \to \cL_1 \cup \{ -\infty \}$,
$\beta': \cL_2 \to \cL_1 \cup \{ +\infty \}$.

Given  $x, y$ in $\cL_i$ let $[x,y]$ be the closed interval
in $\cL_i$ with endpoints $[x,y]$.
Similarly we define $(x,y), (x,y]$ and $[x,y)$.

Notice the following fact: given $\gamma$ in $\pi_1(M)$ then
$\gamma$ preserves the transversal orientation to $\wcF_1$
if and only if $\gamma$ preserves the orientation of $\cL_1$.
Obviously the same holds for $\cF_2$.

The first property is very easy:

\begin{lemma} \label{lem.transorient}
Let $\gamma$ in $\pi_1(M)$. Then $\gamma$
preserves the transversal orientation to $\wcF_2$ if and
only if $\gamma$ commutes with $\alpha$.
Similarly for $\beta$. In addition $\gamma$ preserves
the transversal orientation of $\wcF_1$ if and only if
$\gamma$ commutes with $\alpha'$, or equivalently if and only 
if $\gamma$ commutes with $\beta'$.
\end{lemma}

\begin{proof}
We only prove the first statement.
First think of $\gamma$ acting on 
$\cL_2 \cup \{ -\infty, +\infty \}$. If $\gamma$ preserves the transversal
orientation to $\wcF_2$ then $\gamma(-\infty) = -\infty$, 
otherwise $\gamma(-\infty) = +\infty$.

Suppose that $\gamma$ preserves the transversal orientation
to $\wcF_2$ and let $I$ be the interval of leaves of $\wcF_2$
intersected by $x$ in $\cL_1$, $I$ is an interval in $\cL_2$.
Since $\gamma$ preserves
the orientation from $I$ to $\gamma(I)$ by hypothesis, this implies 
that $\gamma(\alpha(x)) = \alpha(\gamma(x))$, which includes
the case that $\alpha(x) = -\infty$.

Conversely if $\gamma$ reverses the transversal orientation
to $\wcF_2$, then $\gamma$ is orientation reversing
from $I$ to $\gamma(I)$. This implies that
$\gamma(\alpha(x)) = \beta(\gamma(x))$, which includes
the case that $\alpha(x) = -\infty$, implying $\beta(\gamma(x))
= +\infty$.
\end{proof}

The main property we prove in this section is the following Proposition, which is merely an adaptation of the arguments in section $4$ of \cite{Ba1} extended to the case of minimal foliations, not necessarily Anosov. 

\begin{proposition}\label{pro.alphafini}
Let $\cF_1, \cF_2$ be minimal, $\rrrr$-covered, transverse
foliations. Then  either $\alpha$ is constant equal to $-\infty$,
or $\alpha$ is a homeomorphism from $\cL_1$ to $\cL_2$.
The analogous result works for $\beta, \alpha', \beta'$.
\end{proposition}

\begin{proof}
The proofs for $\alpha, \beta, \alpha', \beta'$ are analogous, we do the
proof for  $\alpha$.

If $\alpha$ is identically $-\infty$ the result is achieved.
Assume there is $z$ in $\cL_1$ so that $\alpha(z)$ is not
$-\infty$ and we will prove that $\alpha$
is always finite and a homeomorphism from $\cL_1$ to $\cL_2$.
To start  the analysis  we assume that there is
also $x$ in $\cL_1$ so  that $\alpha(x) = -\infty$.
We can assume  $x <_1 z$ without loss of generality. Since the action of $\pi_1(M)_0$ on $\cL_1$ is minimal there is 
a deck translate $w$ of $x$ so that $w >_1 z$.
Since  $\alpha(x) = -\infty$, let $t$ in $\cL_2$  with
$t \cap  x  \not = \emptyset$ and $t  <_2 \alpha(z)$.
Since $\alpha(w) = -\infty$, we may assume, by taking
$t$ smaller (in $\cL_2$) if necessary, that $t$ also
intersects  $w$. 
Now notice that  $x <_1 z <_1 w$ so $z$ separates $x$ from $w$.
Since $t$ intersects both $x, w$ it now follows that $t$
intersects $z$  as well. This contradicts the fact
that $\alpha(z) >_2 t$.

We have proved that if $\alpha$ is not identically $-\infty$,
then $\alpha$ is never $-\infty$, and hence   it is always finite.

Suppose from now on that $\alpha$ is always finite.

\begin{claim} If $\alpha$ is  not identically $-\infty$, then
$\alpha$ cannot  be constant in any non  trivial interval.
\end{claim}

\begin{proof}
Suppose that $\alpha$ is constant in the non empty open interval $I$ in $\cL_1$.
For any $x$ in $\cL_1$, since the action is minimal, there is  $\gamma$ in 
$\pi_1(M)_0$ so that $\gamma(x)$ is in $I$. Since the map $\alpha$ is
equivariant (because 
$\pi_1(M)_0$ preserves orders), it follows there is an open set containing  $x$ so that
$\alpha$ is constant. Hence every point in $\cL_1$ has a neighborhood
where $\alpha$ is constant. For  any $x, y$ in $\cL_1$, the
interval $[x,y]$ is compact, therefore it follows  that $\alpha$
is constant in $[x,y]$. This would imply that $\alpha$ is
constant and not $-\infty$, which is impossible, again by equivariance of $\alpha$.
The claim follows.
\end{proof}

We  remark that minimality is used  to prove this  claim. 
In fact it is very easy to produce counterexamples if
minimality is not assumed, for example by blowing up
a leaf of (say) $\cF_1$ into an interval of leaves.

\begin{claim}
For any $t$ in $\cL_2$  there are at most two points in $\cL_1$
with image  $t$ under $\alpha$.
\end{claim}

\begin{proof}
Suppose there are $x, y, z$ in $\cL_1$ with $x <_1 y <_1 z$ and
all with image $t$ under $\alpha$.
Any $w >_2 t$ and sufficiently close to $t$ will intersect
$x, y$ and $z$.  It follows that any $v$ in $(x,y)$
intersects $w$ $-$   again because $v$ separates $x$  from $y$
in  $\mt$.  As a consequence  $\alpha(v) \leq_1 \alpha(x) = t$.
The previous claim shows that $\alpha$ cannot be constant in 
$[x,y]$ so there is $v$ in $(x,y)$ so that $\alpha(v) <_2 t$.
Similarly there is $u$ in $(y,z)$ so that $\alpha(u) <_2 t$.

Notice  that the arguments above imply that $v, u$ both intersect
$w$ in  $\cL_2$ with $w >_2 \alpha(x)$ and sufficiently close
to $\alpha(x)$. Since $\alpha(v), \alpha(u)$ are both $<_2 \alpha(x)$,
it follows that $u, v$ both intersect some $q <_2 \alpha(x)$.
Since $y$ separates  $u$ from $v$, then $y$ also intersect $q$,
which is $<_2  \alpha(x) = t$. It follows that $\alpha(y) <_2 t$,
contradiction to assumption.
\end{proof}

\begin{claim}\label{claim-monotone} $\alpha$ is weakly monotone.
\end{claim}

\begin{proof}
Suppose by way of contradiction  that $\alpha$ is not weakly
monotone. Then there are $x, y, z$ in $\cL_1$ with
$x <_1 y <_1 z$ and either  both $\alpha(y) <_2 \alpha(x),
\alpha(y) <_2 \alpha(z)$ hold, or  both  $\alpha(y)  >_2 \alpha(x),
\alpha(y) >_2 \alpha(z)$. (See figure~\ref{fig.claim}.)

We consider the first  possibility, and show it leads to a contradiction.
One  can  deal with the  second possibility in the same way.
Assuming $\alpha(y) <_2 \alpha(x), \alpha(y) <_2 \alpha(z)$, choose
$t$ in $\cL_2$ intersecting $y$ and with $t <_2 \alpha(x),
t <_2 \alpha(z)$.

Let $\pi_1(M)_1$ be a subgroup of order at most $2$ of $\pi_1(M)_0$
so that elements in $\pi_1(M)_1$ preserve the order $<_1$.
Again $\pi_1(M)_1$ acts minimally on $\cL_1$.

Let $\gamma$ be an element of $\pi_1(M)_1$ such that $\gamma(x) >_1 z$. It follows that
$x <_1 y <_1 z <_1 \gamma(x) <_1 \gamma(y) <_1 \gamma(z)$ (this uses
that $\gamma$ preserves the order in $\cL_1$).
In addition $t$ intersects $y$ and $\gamma(t)$ intersects
$\gamma(y)$.

\begin{figure}[ht]
\begin{center}
\includegraphics[scale=0.65]{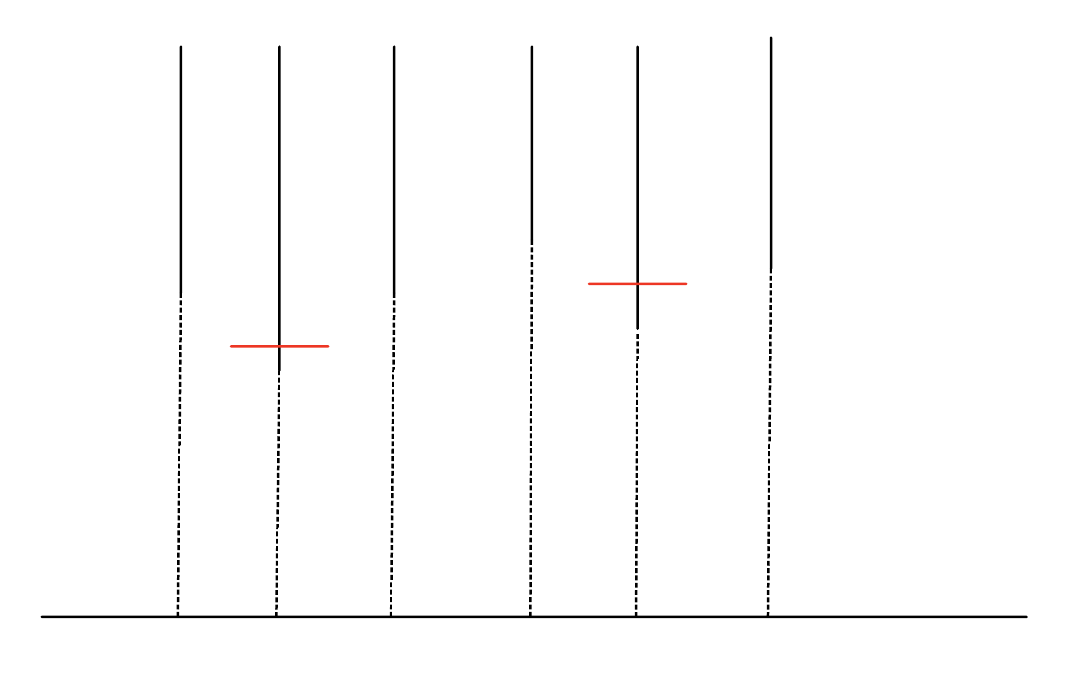}
\begin{picture}(0,0)
\put(-290,15){{\small $x$}}
\put(-255,15){{\small $y$}}
\put(-220,15){{\small $z$}}
\put(-180,15){{\small $\gamma x$}}
\put(-145,15){{\small $\gamma y$}}
\put(-105,15){{\small $\gamma z$}}
\put(-287,120){{\small $\alpha(x)$}}
\put(-255,95){{\small $\alpha(y)$}}
\put(-220,125){{\small $\alpha(z)$}}
\put(-180,135){{\small $\alpha(\gamma x)$}}
\put(-145,105){{\small $\alpha(\gamma y)$}}
\put(-105,135){{\small $\alpha(\gamma z)$}}
\put(-235,105){{\small $t$}}
\put(-125,125){{\small $\gamma t$}}
\end{picture}
\end{center}
\vspace{-0.5cm}
\caption{{\small Depiction of the proof of Claim \ref{claim-monotone}.}}\label{fig.claim}
\end{figure}
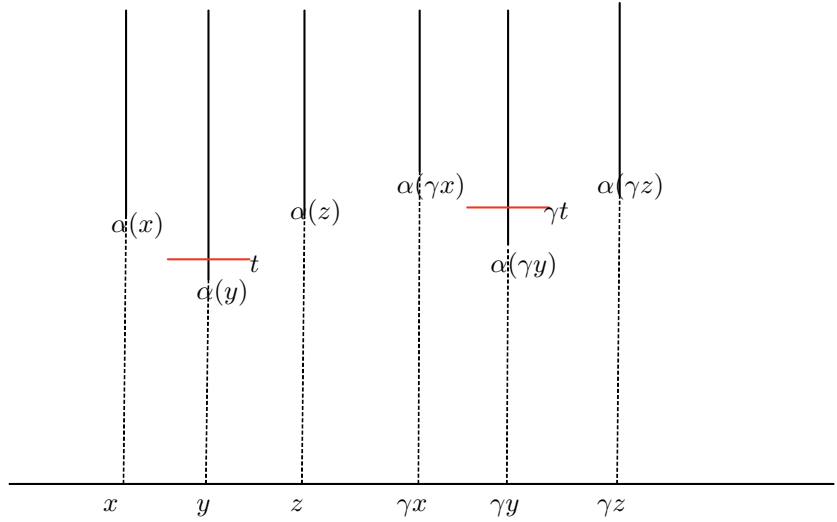

Notice that $t$ does not intersect $z$ and $\gamma(t)$ does
not intersect $\gamma(x)$. It follows that $z, \gamma(x)$
separate $t$ from $\gamma(t)$ and hence
$t, \gamma(t)$ are distinct.
There is $w >_2 t$ intersecting $z$.
In particular if $v <_2 t$ then $t$ separates $v$ from $z$.
Since $\gamma(t)$ is separated from $t$ by $z$, the above
fact implies that $t$ separates $v$ from $\gamma(t)$.
It follows that $\gamma(t) >_2 t$.

Now switch the roles of $t, \gamma(t)$. The
same argument as above implies that $\gamma(t) <_2 t$.
This is impossible.
\end{proof}

Now we can finish the proof of Proposition \ref{pro.alphafini}.
Since   $\alpha$ is weakly monotone  and cannot be constant
on any non trivial interval, then $\alpha$ is (strictly) increasing
or decreasing. If necessary, switch the orientation on $\cL_2$ so
that $\alpha$ is strictly increasing.

The last property to prove is that $\alpha$ is surjective to $\cL_2$. 
Now we use the map $\beta': \cL_2 \to \cL_1 \cup \{ +\infty \}$
previously defined.
The same proof up to this point applied to $\beta'$ shows
that $\beta'$ is weakly monotone.

By the choice of orientation of $\cL_2$ it follows that
$\beta'(\alpha(x))  \leq_1 x$.
To see this first notice that 
$\alpha(x) \cap x = \emptyset$. If $\alpha(x)$ intersects
$z >_1 x$ then $\alpha(z) <_2 \alpha(x)$ contradicting the
choice of orientation in $\cL_2$.
Therefore $\beta'$ is not
always equal to $\infty$, and so by the first part of the 
proof applied to $\beta'$, it follows
that  $\beta'$ is always finite. 

We want to show that for any $x$ in $\cL_1$, then $\beta'(\alpha(x)) = x$.
Suppose by way of contradiction that there is $x$ in $\cL_1$ so that 
$v := \beta'(\alpha(x)) \not = x$.
This leads to a contradiction as follows. 
We have already proved that if $y \in (v,x)$, then $\alpha(y) <_2 \alpha(x)$. Let $t \in \cL_2$ with $\alpha(y) <_2 t <_2 \alpha(x)$,
and $t$ intersects $y$. The last property implies that 
$\beta'(t) >_1 y$. But
$\alpha(x) >_2 t$ and $\beta'(\alpha(x)) = v <_1 y$.
This contradicts that $\beta'$ is weakly monotone increasing.

This contradictions implies that $\beta' \circ \alpha = id_{\cL_1}$.
We already know that $\beta'$ is never $+\infty$, so now switching
the roles of $\alpha$ and $\beta'$ we obtain that
$\alpha \circ \beta' = id_{\cL_2}$.

It follows that if $\alpha$ attains a finite value then
$\beta' = \alpha^{-1}$ and $\alpha$ is a homemorphism from
$\cL_1$ to $\cL_2$.
The same happens if $\beta'$ assumes a finite value.
In an analogous way if $\beta$ assumes a finite value then
$\beta, \alpha'$ are homeomorphisms and $\alpha' = \beta^{-1}$.

This finishes the proof of the proposition.
\end{proof}

\begin{proof}[Proof of Theorem \ref{thm.one}] Let  $\cF_1, \cF_2$ be two foliations satisfying the hypothesis of Theorem \ref{thm.one}. 
Up to switching the foliations we assume that $\cF_2$ is 
transversely orientable, in other words the action of $\pi_1(M)$
on $\cL_2$ preserves orientation. 
Now suppose that 
one leaf of $\wcF_1$ does not intersect every leaf of $\wcF_2$. Then, up to changing the order $<_2$, we can assume without loss of 
generality that the map $\alpha$ defined in Definition \ref{def.alphabeta} has finite value.  According to Proposition \ref{pro.alphafini}, $\alpha$ is a homeomorphism from $\cL_1$ to
$\cL_2$.
In addition, by hypothesis, for any $\gamma$ in $\pi_1(M)$ then $\gamma$ 
preserves the order in $\cL_2$, so Lemma \ref{lem.transorient} 
implies that for any 
$\gamma$ in $\pi_1(M)$, then
$$\gamma \circ \alpha \  = \ \alpha \circ \gamma.$$
\noindent
In other words $\alpha$ is a conjugacy between the actions of
$\pi_1(M)$ on $\cL_1$ and $\cL_2$.

The only property left to prove is that $\cF_1$ is also transversely
orientable. Again without loss of generality assume that $\alpha$
is finite. Then $\beta' = \alpha^{-1}$. Since $\alpha$ commutes
with every $\gamma$ in $\pi_1(M)$, then $\beta' = \alpha^{-1}$
implies that $\beta'$ commutes with every $\gamma$ in $\pi_1(M)$
as well. Now using Lemma \ref{lem.transorient} in the other direction,
this implies that $\cF_1$ is transversely orientable.

This finishes the proof of the theorem.
\end{proof}

As already explained, if $M$ has dimension $3$ and if the foliations are uniform, by \cite[Proposition 2.2]{Th5}, this implies that the foliations are isotopic as defined in \S~\ref{sub.homotopic}. Note that even in hyperbolic 3-manifolds there exist $\rrrr$-covered foliations which are not uniform \cite[Ex. 9.18]{Cal2}.



\section{The case of translation foliations and product $\rrrr$-covered Anosov foliations}\label{s.translationfoliations}

In this section, we consider two special cases: the case (in dimension $3$) where the two foliations are product Anosov, and the case (in any dimension) where one of the foliation is a translation foliation. These cases have already been considered in \cite{Ma-Ts}, but sometimes under the hypothesis that both foliations are sufficiently differentiable. Here we see how to adapt their arguments, when needed, in the case where $\cF_1$ and $\cF_2$ are merely continuous.

\subsection{The case where one foliation is a translation foliation}

Here we do not restrict the dimension of  $M$ or the type of the
second foliation.

Notice that a translation foliation is necessarily transversely
orientable.

Recall that we have a local homeomorphism (the developing map) $\cD_{\cG}: \cL_{\cG} \to \cL_1 \times \cL_2$ that commutes with the action of $\pi_1(M)$ (see Definition \ref{developingmap}).

The proposition below is a slight generalization of \cite[Proposition 2.3]{Ma-Ts}.

\begin{proposition}\label{pro;trans+foli}
    Let $(\cF_1, \cF_2)$ a pair of transverse $\rrrr$-covered foliations on a closed manifold $M$ with intersection foliation $\cG = \cF_1 \cap \cF_2.$ Assume that $\cF_2$ is a translation foliation. Then, $\cD_{\cG}: \cL_{\cG} \to \cL_1 \times \cL_2$ is a homeomorphism onto the entire $\cL_1 \times \cL_2$. Moreover, $\cD_{\cG}$ conjugates the actions of $\pi_1(M)$ on $\mathcal L_{\mathcal G}$ and $\cL_1 \times \cL_2$.  
\end{proposition}

By conjugation we mean a topological conjugation, which is a homeomorphism between the spaces that conjugates the actions.

\begin{proof}
    It is simply an adaptation of the proof of 
Lemma \ref{le:translationR} in which we have constructed a map $F^\varepsilon: \mt \to \cL_2$ for $\varepsilon$ sufficiently small. 
Let $d: \mt \to \rrrr$ be the developing map of the translation foliation $\cF_2$ (see Definition \ref{def:translation}). The same techniques allow one to construct a map $F_{\cG}^\varepsilon: \mt \to \cL_{\cG}$ associating to every element $\tilde p$ of $\mt$ the unique leaf $g$ of $\wcG$ in the leaf $F_{\tilde p}$ of $\wcF_1$ containing $\tilde p$ passing $\epsilon_+$-near to $\tilde p$ and satisfying $d(g) = d(\tilde p) + \varepsilon.$ The same proof works and allow to show that the restriction of $\wcG$ to the leaf $F_{\tilde p}$ has a leaf space homeomorphic to $\rrrr$ and such that the restriction of $d$ to this leaf space is a homeomorphism onto $\rrrr$. 
The same arguments as in Lemma \ref{le:translationR} show that each leaf
of $\wcF_1$ intersects every leaf of $\wcF_2$. 
It follows that the intersection between any leaf of $\cF_1$  and any leaf of $\cF_2$ is a single leaf of $\cG,$ which is what we wanted to show. The map $\cD_{\cG}$ commutes with the actions of $\pi_1(M)$ by definition. 
\end{proof}

An interesting corollary of this proposition is that one can prove that in some cases the intersection foliation $\cG$ does not depend on the way we isotope one foliation to make it transverse to the other, at least up to homotopic equivalence.

\begin{corollary}
     Let $\cF_1$ be a $\rrrr$-covered codimension one foliation on $M$ which is a classifying space, and let $\cF_2$, $\cF'_2$ be a pair of translation foliations on $M$, isotopic one to the other (in the sense of \S~\ref{sub.homotopic}), and both transverse to $\cF_1.$ Then, the intersection foliations $\cF_1 \cap \cF_2$ and $\cF_1 \cap \cF'_2$ are homotopically equivalent.
\end{corollary}

\begin{proof}
    According to Proposition \ref{pro;trans+foli} the action of $\pi_1(M)$ on the orbit spaces of $\cG = \cF_1 \cap \cF_2$ and $\cG' = \cF_1 \cap \cF_2$ are both topologically conjugate to the action of $\pi_1(M)$ on the product $\cL_1 \times \cL_2$.  It implies that their holonomy groupo\"{\i}ds are isomorphic.

    Let $g$ be a leaf of $\wcG$ and $F$ the leaf of $\wcF_1$ containing $g$. According to Proposition \ref{pro;trans+foli}, the restriction of $\cD$ to $F$   
    is a fibration over $\rrrr$, with fibers homeomorphic to $g.$ Hence the injection $\pi_1(g) \hookrightarrow \pi_1(F)$ of the fundamental groups is an isomorphism.
     Then, the holonomy for $\wcF_1$ along such a loop is trivial if and only if the holonomy along this loop for $\wcG$ is trivial (because the $\cL_2$ component of this holonomy is already trivial).

     Therefore, the kernel $K$ of the holonomy morphism of $g$ is also the kernel of the holonomy morphism of $F.$ 
     The covering $F_K$ of $F$ associated to $K$ is therefore a $g_K$ bundle over $\rrrr$, where $g_K$ is the holonomy covering of $g.$ Since $F_K$ is contractible, the exact homotopy sequence of the fibration shows that all homotopy groups of $g_K$ vanish. Therefore, $g_K$ is contractible, and $\cG$ is a classifying space for its holonomy groupo\"{\i}d. 

    The same holds for $\cG'$. They are both classifying foliated spaces and have the same holonomy groupo\"{\i}d: they are homotopically equivalent.    
\end{proof}
   
This corollary generalizes \cite[Propositions 2.2 and 2.3]{Ma-Ts} removing the assumption that $M$ is 3-dimensional and the transverse regularity of the foliations. Observe that in dimension 3, the ambient manifold is, up to finite covering, a torus bundle over the circle, and the homotopy equivalence can be shown to be a topological equivalence. 

We end up this section with a remark that will be useful later. Observe that translation foliations in (orientable) $2$-dimensional manifolds are topologically conjugate to linear foliations on the torus, and it is quite well-known that such a linear foliation is not isotopic (in the sense of \S~\ref{sub.homotopic}) to a foliation transverse to itself. The following lemma generalizes this fact.

\begin{lemma}\label{le:2transv}
    Let $(\cF_1, \cF_2)$ be pair of translation foliations transverse one to the other.  Then $\cF_1$ is not isotopic to $\cF_2$.
\end{lemma}

\begin{proof}
    According to the Proposition \ref{pro;trans+foli} the action of $\pi_1(M)$ on the orbit space $\cL_{\cG}$ is topologically conjugate to the action of $\pi_1(M)$ on $\cL_1 \times \cL_2$. Suppose that $\cF_1$ is isotopic to $\cF_2$. Then the actions of $\pi_1(M)$ on $\cL_1$ and $\cL_2$ are topologically conjugate one to the other, and both conjugate to an action by translations on the real line. More precisely, there is an identification of
    $\cL_{\cG}$ with $\rrrr^2$ and a morphism $r: \pi_1(M) \to \rrrr$ such that the action of an element $\gamma$ of $\pi_1(M)$ corresponds to the diagonal translation     
    $(x,y) \mapsto (x + r(\gamma), y + r(\gamma)).$ Then, the composition of $\eta: \mt \to \cL_{\cG} \approx \rrrr^2$ and the map $(x,y) \mapsto y-x$ induces a 
    $\pi_1(M)$-invariant topological submersion from $\mt$ to $\rrrr,$ hence a submersion from $M$ to $\rrrr.$ It is a contradiction since $M$ is compact: such a map admits a maximum, and therefore cannot be a submersion.
\end{proof}

\subsection{Product Anosov foliations}
\label{s.prod}

In this subsection $(\cF_1, \cF_2)$ is a bifoliation made of minimal foliations on the $3$-manifold  $M(A)$ for some hyperbolic element $A \in$ SL$(2, \zzzz)$ (cf. section \ref{sub.product}). We will assume that the foliations are orientable and transversally orientable. This manifold is the ambient manifold of the suspension of $A,$ considered as a linear diffeomorphism of the torus $\rrrr^2/\zzzz^2$. We denote by $\Phi_A$ the suspended Anosov flow with constant return time, by $\cF^s$, $\cF^u$ the weak stable and unstable foliations of $\Phi_A,$ and by $\cF^{ss}_A$, $\cF^{uu}_A$ the strong stable and unstable foliations. Observe that the closures of the leaves of $\cF^{ss}_A$ and $\cF^{uu}_A$ are fibers of the fibration over the circle. (We emphasize also that the strong stable and unstable foliations are sensible to reparametrizations of the flow, here, we are considering the ones associated to the constant time suspension, and it is well known that for small reparametrizations of the flow, the strong stable and strong unstable foliations become minimal.)

According to a result of E. Ghys and V. Sergiescu (\cite{Gh-Se}) any codimension one minimal foliation $\cF$ on $M(A)$ is topologically conjugated to either $\cF^s$ or $\cF^u$ of $\Phi_A$ (actually, they just need that there is no compact leaves).

Actually, Ghys and Sergiescu assume some regularity of the foliation, and prove that the conjugacy is $C^2$, but this hypothesis can be removed. More precisely, S. Matsumoto proves in \cite[Theorem 3]{Mat2} that the leaf space $\cL$ is homeomorphic to the real line, on which $\pi_1(M(A))$ acts by affine transformations of the form $x \mapsto \lambda(\gamma).x + a(\gamma)$ (removing the regularity does need the minimality assumption). 

Recall that We have a split exact sequence:
$$0 \to H \to \pi_1(M(A)) \to \zzzz \to 0 $$
where $H$ is a normal subgroup isomorphic to $\zzzz^2$ corresponding to the fundamental group of the fibers.

It also follows from \cite{Mat2} that 
the action of $H$ on $\cL$ is a faithfull and free action by translations. It means that the morphism $\lambda: \pi_1(M(A)) \to \rrrr^*$ is trivial when restricted to $H$ (with value $1$), and in addition that the restriction of $a: \pi_1(M(A)) \to \rrrr$ to $H$ is a non-vanishing morphism.

Let $\gamma_0$ be an element of $\pi_1(M)$ projecting to a generator of $\pi_1(M)/H \approx \zzzz$. The action of $\gamma_0$ on $H$  by conjugacy is conjugate to the action of $A$ or its inverse $A^{-1}$ on $\zzzz^2$. This action is not trivial, meaning that the action of $\gamma_0$ on the affine line $\cL$ is not an action by translation, but is of the form $x \mapsto \nu x + a$ with $\nu$ is a positive real number different from $1$. Actually, the relation $\gamma_0 h \gamma_0^{-1} = A(h)$ for every $h \in H$ implies that $\nu$ must be one of the eigenvalues $\lambda$ or $\mu$ of $A$ (here, we identify $H$ with $\zzzz^2$, and $A(h)$ denotes the action of the matrix on vectors with integer coefficients).

Here we have a pair $(\cF_1, \cF_2)$ of transverse foliations, each without compact leaves. For each of them, the linear part of the action of $\gamma_0$ on the leaf space $\cL_i$ is a multiplication by a factor $\nu_i$ which is equal to $\lambda$ or to $\mu.$

There are two cases to consider:
\begin{enumerate}
     \item Either $\nu_1 \neq \nu_2$,
    \item Or $\nu_1 = \nu_2$.
\end{enumerate}

Case $(1)$ is of course possible; it is realized by the pair $(\cF^s, \cF^u)$ of the weak foliations of $\Phi_A.$

Let us show that case $(2)$ can indeed be realized, and that the intersection foliation $\cF_1 \cap \cF_2$ is topologically conjugate to a strong stable foliation $\cF^{ss}$ or $\cF^{uu}$ of the suspension Anosov flow $\Phi_A$ (with constant return time). For that purpose, we need to be more precise than in subsection \ref{sub.defanosov}. (Note that this example is also explained in \cite[\S 6]{FP3} with other consequences in mind.)

Recall from \S~\ref{sub.product} that the leaf spaces $\cL^s$ and $\cL^u$ are both identified with $\rrrr$, and for every element $\gamma$ of $\pi_1(M(A))$ there is an integer $k(\gamma)$ and real numbers $a(\gamma)$, $b(\gamma)$ such that the action of $\gamma$ on $\cL^s$ is $x \mapsto \lambda^{k(\gamma)}x + a(\gamma)$ and the action on $\cL^u$ is given by $y \mapsto \mu^{k(\gamma)}y + b(\gamma)$. Then, $k$ defines a morphism from $\pi_1(M(A))$ to $\zzzz$, with kernel $H$, and the set of all pairs $(a(\gamma), b(\gamma))$ form a lattice in $\rrrr^2.$ Moreover, the integer $k(\gamma)$ is, up to the sign, the projection of $\gamma$ in $\pi_1(M(A))/H\approx \zzzz.$

\begin{proposition}\label{pro.discontinu}
    Let $\Omega$ the open domain in $\cL^s \times \cL^s \times \cL^u$ made of triples $(x,y,z)$ where $y \neq x.$ The action of $\pi_1(M(A))$ is free and properly discontinuous.
\end{proposition}

\begin{proof}
Assume that some element $\gamma$ of $\pi_1(M(A))$ preserves an element $(x, y, z)$ of $\Omega.$ Then we have $x =\lambda^{k(\gamma)}x + a(\gamma)$, $y = \lambda^{k(\gamma)}y + a(\gamma)$, and $z = \mu^{k(\gamma)}z + b(\gamma)$. Hence we have $x-y = \lambda^{k(\gamma)}(x-y),$ therefore $k(\gamma) = 0$ since $x \neq y.$ It follows that $a(\gamma) = b(\gamma) = 0:$ the element $\gamma$ is trivial. The action of $\pi_1(M(A))$ is free.

We consider now a sequence $(x_n, y_n, z_n)$ of elements of $\Omega$ and a sequence of elements of $\pi_1(M(A))$ such that $(x_n, y_n, z_n)$ converges to some element $(x, y, z)$ of $\Omega$ and such that $(x'_n, y'_n, z'_n) = \gamma_n(x_n, y_n, z_n)$ converges to an element $(x', y', z')$ of $\Omega$. Then, $y'_n - x'_n$ converges to $y' -x'$, meaning that $\lambda^{k(\gamma_n)}(y_n - x_n)$ admits a non-zero limit. Since $y_n - x_n$ converges to $y-x$ which is not $0$, it shows that the integers $k(\gamma_n)$ must have the same value $k$ for $n$ sufficiently big.

Now $\lambda^{k}x_n + a(\gamma_n)$ converges to $x'$: it implies that $a(\gamma_n)$ converges to $x'- \lambda^{k}x$. Similarly, one shows  that $b(\gamma_n)$ converges to $z'- \mu^{k}z$. Since the set of pairs $(a(\gamma), b(\gamma))$ is discrete, then $a(\gamma_n),  b(\gamma_n)$ are fixed for $n$ big.
Since the same  is true for $k(\gamma_n)$, it implies that all the $\gamma_n$ have the same value $\gamma$ for $n$ sufficiently big, and that $(x', y', z') = \gamma (x, y, z).$ It finishes the proof of the proposition.    
\end{proof}

Hence, the quotient of $\Omega$ by $\pi_1(M(A))$ is a $3$-manifold. Since we are assuming that $A$ has determinant $1$ and positive eigenvalues, this quotient has two connected components, one corresponding to $y > x$ and the other to $y < x,$ and that we denote by $M^+(A).$

Observe that $M^+(A)$ is a $K(\pi_1(M(A)), 1)$, as $M(A)$, and therefore compact. More precisely, 
$M^+(A)$ and $M(A)$ are homeomorphic since they are both irreducible closed $3$-manifolds with an incompressible surface (see \cite{Wald}). 
The fibration over the circle is actually quite easy to construct: it is the map induced by $(x, y, z) \mapsto \ln(|y-x|).$ Indeed, the level sets of this map are preserved by the abelian normal subgroup $H.$ Further details are left to the reader.

More importantly, every factor defines a codimension $1$ foliation on $M^+(A).$ In other words, the fibers of the map $(x,y,z) \mapsto x$ are the leaves of a foliation $\wcF_x$ in $\Omega$ which is preserved by $\pi_1(M(A))$, hence defines a foliation $\cF_x$ in $M(A).$ One defines similarly foliations $\cF_y$ and $\cF_z.$ These three foliations are transverse one to the other. The leaf spaces $\cL_x$ and $\cL_y$ of $\cF_x$ and $\cF_y$ are both canonically identified with $\cL^s$. It is not hard to check that $\cF_x, \cF_y$ are both isotopic (in the sense of \S~\ref{sub.homotopic}) to $\cF^s.$ 

In summary, $\cF_x$ and $\cF_y$ are  two foliations, each isotopic to $\cF^s$, and transverse one to the other. The intersection foliation $\cG_z = \cF_x \cap \cF_y$ is tangent to the fibers of the fibration over the circle defined above. More specifically they define in each of these tori a linear foliation without compact leaves - actually, 
$\cG_z$ is the strong stable foliation of the suspension Anosov flow. In particular $\cG_z$ has no closed leaves.

The orbit space of $\wcG_z$ is the region of $\rrrr^2$ defined by $y > x.$ Observe that the action of $\pi_1(M(A))$ preserves the hyperbolic metric $\frac{dx^2 + dy^2}{(y-x)^2}$. It follows that $\cG_z$ is transversely hyperbolic. It is the example 6 in the classification of Y. Carri\`ere of Riemannian flows in dimension $3.$ (section I. D of \cite{Car}). 

Last but not least, let $\ell$ be any direction in the plane $\cL^s \times \cL^s \approx \rrrr^2$. Then lines of direction $\ell$ are leaves of some $\pi_1(M(A))$-invariant foliation of $\rrrr^2$, and therefore define a foliation $\cF_\ell$ in $M(A) \approx M.$ 
We thus obtain a $1$-parameter family of foliations in $M(A)$, parameterized by the space $\rrrr\mathbb{P}^1$ of directions, all tangent to $\cG_z,$ and all transverse one to the other (in other words, a \textit{tissu feuillet\'e} in the terminology of \cite{Ghfeuillete}).

When $\ell$ is the direction $\ell_0$ of the diagonal $\Delta = \{ x=y \}$, the associated foliation $\cF_{\ell_0}$ is the one
whose leaves are the fibers of the fibration over the circle. For $\ell \neq \ell_0$, the intersection of the line with the diagonal $\Delta$ provides an equivariant identification between the leaf space of $\cF_\ell$ and $\cL^s.$ Hence $\cF_\ell$ is isotopic to $\cF^s.$

In other words, the foliation $\cF_{\ell_0}$, whose leaves are all tori, is continuously approximated by foliations isotopic to $\cF^s.$

As a final remark, we underline that the $1$-dimensional foliations $\cG_x = \cF_y \cap \cF_z$ and $\cG_y = \cF_x \cap \cF_z$ are both isotopic to the $1$-dimensional foliation defined by the Anosov flow.

And of course, if $\mu$ is positive, one similar construction can be made starting from $\cL^u \times \cL^u \times \cL^s$ instead of $\cL^s \times \cL^s \times \cL^u$, leading to a transverse intersection of $\cF^u$ with a foliation isotopic to it.
The intersection foliation is now isotopic to the {\em unstable} strong foliation of the Anosov flow. 
\medskip

We end this section with the following proposition, establishing that the examples above are unique up to topological equivalence.

\begin{proposition}\label{productproduct}
    Let $(\cF_1, \cF_2)$ be a pair of minimal transverse foliations on $M(A)$. Then, the intersection foliation $\cG = \cF_1 \cap \cF_2$ is orbitally equivalent to 
    one of the following $1$-dimensional foliations:
    \begin{enumerate}
        \item the Anosov foliation $\Phi_A$, 
        \item the strong stable foliation $\cF^{ss}_A$,
        \item or the strong unstable foliation $\cF^{uu}_A$
    \end{enumerate}
\end{proposition}

We recall again that here, the strong stable and strong unstable foliations denote the ones of the constant time suspension. While the Anosov foliation is structurally stable (i.e. if one considers $\Phi' \sim \Phi_A$ then the orbit foliation is orbit equivalent to the one of $\Phi_A$), this is not the case for $\cF^{ss}_A$ and $\cF^{uu}_A$ so these latter cases involve a stronger rigidity. Note in particular that neither $\cF^{ss}_A$ nor $\cF^{uu}_A$ have dense leaves (the closure of each leaf is a two dimensional tori) while $\Phi_A$ has many such leaves as well as periodic ones.

\begin{proof}
As we already observed, it follows from \cite{Mat2} that the leaf spaces $\cL_1$ and $\cL_2$ are both copies of the affine line $\rrrr$ on which $\pi_1(M(A))$ acts affinely. Moreover the abelian subgroup $H$ acts by translations, whereas other elements of $\pi_1(M(A))$ act by $x \mapsto \nu_i^{k(\gamma)}x + a(\gamma)$, where $k(\gamma)$ is not null.

Observe that $a(\gamma)$ is not necessarily the one describing the action on $\cL^s$ or $\cL^u,$ it could be twisted by some automorphism of $\pi_1(M(A)).$

The phase space $\cL_1 \times \cL_2$ is identified with $\rrrr^2$ so that $\pi_1(M(A))$ acts affinely and diagonally.  
It follows that $\cG$ is transversely affine. 
Observe that $H$ acts by translations, but it is not clear yet that it forms a lattice of translations.

There are two possibilities: the first one is the case where $\nu_1 = \nu_2$. In this case, every direction $\ell$ in the plane $\cL_1 \times \cL_2 \approx \rrrr^2$ is preserved by the holonomy, that therefore preserves the foliation of $\cL_1 \times \cL_2$ whose leaves are lines of direction $\ell.$ The pull-back of this foliation is a foliation $\cF_\ell$ in $M.$ These foliations form a \textit{tissu feuillet\'e}: according to \cite[page 149]{Ghfeuillete}, the flow $\cG$ is transversely Riemannian since $M$ is not homeomorphic to $\mathbb{S}^2 \times \mathbb{S}^1.$ In this case, according to the classification of Riemannian flows by Y. Carri\`ere (see \cite[page 132]{Ghfeuillete}, \cite{Car}),  
the foliation $\cG$ is topologically conjugate to either the strong stable foliation $\cF^{ss}$ or to the strong unstable foliation $\cF^{uu}$. 
We emphasize that these results do not use regularity 
of the foliation. These are cases $(2)$ and $(3)$ in the statement.

The other possibility is that $\nu_1 \neq \nu_2$. Let us say without loss of generality $\nu_1 = \lambda$ and $\nu_2 = \mu$. Since $\lambda\mu = 1$ it follows that the action of $\pi_1(M(A))$ preserves a volume form on $\cL_1 \times \cL_2.$ 

Moreover, in this case, the actions of $\cL_1$ and $\cL_2$ are not conjugate one to the other. Recall the maps $\alpha$ and $\beta$ introduced in Definition \ref{def.alphabeta}. If one of them were finite,
then with the transverse orientability hypothesis, Theorem \ref{thm.one} 
would imply that the two actions are conjugate. It follows that 
$\alpha = -\infty$, $\beta = +\infty$ always, that is,
every leaf of $\wcF_1$ intersects every leaf of $\wcF_2.$

Consider now the diagonal action of $H$ on $\cL_1 \times \cL_2:$ it is an action by translations. This action is free, and defines some abelian group of translations $T$ isomorphic to $\zzzz^2.$ 

Let $h_1$, $h_2$ be two generators of $T$. If they are colinear, then they are contained in a unique line $\Delta$ of $\rrrr^2.$ The line $\Delta$ in $\rrrr^2$ must be invariant by the linear part of the affine action of $\gamma_0$ (recall that $\gamma_0$ is an element of $\pi_1(M(A))$ projecting on a generator of $\pi_1(M(A))/H$). It is a contradiction since $\gamma_0$ is a dilatation of factor $\lambda$ on $\cL_1$ and a contraction of factor $\mu$ on $\cL_2$ (observe also that $\Delta$ cannot be vertical or horizontal since $H$ acts freely on $\cL_1$ and $\cL_2$:  if say $\Delta$
is $\{ x \} \times \rrrr$, then $H$ preserves $x$).

Therefore, $h_1$ and $h_2$ are not colinear: they form a basis of $\rrrr^2.$
It means that $H$ is a lattice in the space of translations on 
$\cL_1 \times \cL_2 \approx \rrrr^2$. It follows that the action of $\pi_1(M(A))$ on $\cL_1 \times \cL_2$ is topologically conjugate to its action on the orbit space $\cL^s \times \cL^u$ of the Anosov flow $\Phi_A$. 

We then can use the argument in \cite[\S 3]{Ma-Ts} verbatim, that leads to the conclusion that the map $\cD_{\cG}: \cL_{\cG} \to \cL_1 \times \cL_2 \approx \cL^s \times \cL^u$ is a trivial fibration with fiber $\rrrr,$ hence that $\cG$ is orbitally equivalent to the suspension Anosov flow $\Phi_A.$ This is case $(1)$ of our statement, whose proof is now complete.
\end{proof}

\section{Group invariant monotone graphs in the phase space}
\label{sec:invariant}

In this section we do preparatory  material for the
analysis of Theorem \ref{thm.three}. 
We consider a pair $(\cF_1, \cF_2)$ of transverse, $\rrrr$-covered, minimal codimension one foliations. We do not restrict to $M$ of dimension $3$
and we do not restrict to 
$\cF_i$ being Anosov foliations.

We will assume that all foliations are orientable and transversally orientable. This assumption will be lifted in \S~\ref{s.lastone}. 
Finally, we assume that there is  a homeomorphism $f: \cL_1 \to \cL_2$
conjugating the actions of $\pi_1(M)$ on $\cL_1$
and  $\cL_2$: for any $x \in \cL_1$ and $\gamma \in \pi_1(M)$,
then $f(\gamma(x)) = \gamma(f(x))$.

Such a conjugacy might be non-unique, and the aim of this section is to study the space $\cH_0$ of conjugacies:

$$\cH_0  = \{  f: \cL_1 \to \cL_2, \ \ \forall \gamma \in \pi_1(M), \forall x \in \cL_1, f(\gamma(x)) = \gamma(f(x))\}$$

We select on element $f_0$ in $\cH_0$. 
Put orders in $\cL_1, \cL_2$   so that $f_0$ is  an
increasing  homeomorphism.

Actually, sometimes we prefer to consider elements of $\cH_0$ through their graphs, that we call 
\textit{$\pi_1(M)$ invariant  monotone  graphs}.
Indeed, a homeomorphism $h$ is a topological conjugacy between the actions of $\pi_1(M)$ 
on $\cL_1$ and $\cL_2$ if and only if the graph of $h$ is invariant under the diagonal  action $\gamma((x,y))
= (\gamma(x), \gamma(y))$ on
$\cL_1 \times \cL_2$.

We denote by $\cC$ the set of $\pi_1(M)$-invariant monotone graphs:

$$\cC  = \{  \delta, \ \ {\rm which \ is \ a \ monotone \ graph \ 
and  \ } \pi_1(M)
\ {\rm  invariant} \}$$

Recall that we have selected orders on $\cL_1$ and $\cL_2$ so that $f_0$ is increasing. The first Lemma shows that it will be the case for any other conjugacy.

\begin{lemma}
    Let $h$ be an element of $\cH_0$. Then, $h$ is increasing. 
\end{lemma}

\begin{proof}
    Assume that $h$ is decreasing. Then, the graphs of $h$ and $f_0$ intersect. Moreover, the intersection is a single  point. 
    This point must be fixed by the diagonal action of $\pi_1(M)$, and this is impossible since the action of $\pi_1(M)$ on $\cL_1$ is minimal.
\end{proof}

\begin{lemma}\label{lemma52} Suppose that  two invariant monotone graphs $\delta, \theta$  in $\cC$
intersect. Then $\delta = \theta$.
\end{lemma}

\begin{proof} Let $f, g$ be the functions whose graphs are
$\delta, \theta$. Since $\delta, \theta$ intersect, there is $x$ in $\cL_1$
so that $f(x) = g(x)$. 
Since $f, g$ are group equivariant,
then

$$f(\gamma(x)) = \gamma(f(x)) = \gamma(g(x)) = g(\gamma(x)).$$ 

\noindent
The set of such $\gamma(x)$ is
dense in  $\cL_1$ since $\cF_1$  is minimal. Since $f, g$ are
continuous, it follows that  $f = g$.
\end{proof}

We now produce a group acting on $\cC$. 
We use our selected element $f_0$ of $\cH_0.$ 
Let $\cH$ be the set of homeomorphisms of $\cL_1$ onto itself of the 
form $(f_0)^{-1} \circ f: \cL_1 \to \cL_1$ where $f$ is any element of $\cH_0.$

\begin{lemma} \label{graphgroup}
The  set $\cH$ is a subgroup of the group of homeomorphisms of
$\cL_1$. It is an abelian group,
which acts freely on $\cL_1$.
\end{lemma}

\begin{proof}
Obviously, $\cH$ contains the identity map $(f_0)^{-1} \circ f_0$,
as $f_0$ is in $\cH_0$.
Consider compositions: let $(f_0)^{-1} \circ f$ and
$(f_0)^{-1} \circ g$ be two elements of $\cH$, where $f, g$ are in $\cH_0$.
The composition is $(f_0)^{-1} \circ f \circ (f_0)^{-1} \circ g$.
So one property to prove is that $f \circ (f_0)^{-1} \circ g$
is in $\cH_0$. Notice that $f_0, g, f$ are $\pi_1(M)$
equivariant. Given  $y$ in $\cL_2$ let $x = (f_0)^{-1}(y)$.
Then 

$$f_0(\gamma x) = \gamma (f_0(x)), \  {\rm which \ implies}  
\ \ f_0^{-1}(\gamma (y)) = \gamma(f_0^{-1}(y)).$$ 

\noindent
Obviously the  same holds for $f$ and $g$ in the place of $f^{-1}_0$.
Now compute with $f \circ (f_0^{-1})  \circ g$ applied to
$\gamma(x)$. First $g(\gamma x) = \gamma(g(x))$.
Then

$$(f_0)^{-1}(\gamma (g(x))) \ = \ 
\gamma(f_0^{-1} \circ g(x)).$$

\noindent
Finally 
$$f(\gamma(f_0^{-1} \circ g)(x) \ = \ \gamma(f \circ f_0^{-1} \circ
 g)(x),$$

\noindent
so $f \circ f_0^{-1} \circ g$ is also in $\cH_0$
$-$ notice that $f \circ f^{-1}_0 \circ g$ is an
increasing homeomorphism from $\cL_1$   to $\cL_2$.
It follows that $\cH$ is closed under composition.
As for inverses, if $(f_0)^{-1} \circ  f$ is in $\cH$, then
its inverse is $f^{-1} \circ f_0$, and this is
$(f_0)^{-1} \circ h$, where $h = f_0 \circ f^{-1} \circ f_0$.
As was done above one shows that $h$ is in $\cH_0$, so
$\cH$ is also closed under inverses.
This shows that $\cH$ is a group of homeomorphisms of $\cL_1$.

Given any $x$ in $\cL_1$ and $h$ in $\cH$, then  $h = (f_0)^{-1}  \circ f$
for  some $f$ in $\cH_0$. If $f, f_0$ are distinct, then the
graphs of $f, f_0$ are disjoint. In that  case it follows that
$h(x) \not = x$. In other words, $\cH$ acts freely on $\cL_1$
which is homeomorphic to $\rrrr$. It follows from H\"older's Theorem (\cite{hector}) that  $\cH$ is
abelian.
\end{proof}

H\"older's Theorem actually gives a more complete answer. A group of homeomorphisms acting freely on the real line is not only abelian, but semi conjugated 
to an action by translations. This allows us to prove the following
stronger result:

\begin{proposition}\label{pro:notdiscrete}
The set of  $\pi_1(M)$
invariant monotone graphs forms a discrete set of curves in $\cL_1 \times \cL_2$.
\end{proposition}

\begin{proof}
Assume that the collection of $\pi_1(M)$
invariant monotone graphs forms a non-discrete set of curves in $\cL_1 \times \cL_2$.
According to H\"older's Theorem, the action of $\cH$ on $\cL_1$ is semi conjugated to an action by translations. 

This action of $\cH$ on $\cL_1$ has a unique minimal set in $\cL_1$
(see \cite{hector}), which we denote by $\cE$.
We show that $\cE$ 
 is preserved by the action of $\pi_1(M)$.
Let $\gamma$ be an element of $\pi_1(M)$. 
Then $\gamma(\cE)$ is certainly closed. 
Given $f_0^{-1} \circ f$ and $z$ in $\cE$,
then 

$$f_0^{-1} \circ f(\gamma z) \ = \ f_0^{-1}(\gamma(f(z))) \ = \ 
\gamma(f_0^{-1} \circ f(z)).$$

\noindent
In particular this shows that the actions of $\pi_1(M)$
and $\cH$ on  $\cL_1$ commute.
Now $f_0^{-1} \circ f(z)$ is in  $\cE$ (as $\cE$ is $\cH$ invariant),
so $\gamma(f_0^{-1} \circ f(z))$ is in $\gamma(\cE)$. This implies that $\gamma(\cE)$ is
$\cH$ invariant as we wanted to prove.
If $\gamma(\cE)$ is not minimal, then applying $\gamma^{-1}$ 
and the argument above implies that $\cE$ is not minimal, contradiction.
We conclude that by unique minimal set we have that $\gamma(\cE) = \cE$.

 Since the $\pi_1(M)$ action on $\cL_1$ is minimal, the minimal set $\cE$ is the entire $\cL_1$ and
the semiconjugacy is actually a conjugacy. 

After conjugacy we can  assume that $\cH$ is an action by
translations. In addition since  the set of invariant monotone graphs
is not discrete,  then the set of translations is not discrete,
and hence it is a dense set of translations.
For any $\gamma$ in $\pi_1(M)$, then $\gamma$  commutes
with every element of $\cH$ and $\cH$ is a dense set of
translations in the reals $-$ which immediately implies that
$\gamma$ acts as a translation. This is true for any $\gamma$, implying that
$\cF_1$ is a translation foliation (as defined in Definition \ref{def:translation}).
Switching the roles roles of $\cF_1, \cF_2$ this also implies
that $\cF_2$ is a translation foliation.
Since both  $\cF_1$ and $\cF_2$ are translation  foliations,
Proposition \ref{pro;trans+foli} implies that the image
of $\cD$ is the  whole phase space $\cL_1 \times \cL_2$.

The rest of the proof uses ideas in  the proof of Lemma \ref{le:2transv}.
We  have a parametrization of $\cL_1$ as the reals so that
$\pi_1(M)$ acts as a translation. 
We express 
this parametrization by a homeomorphism
$h: \cL_1 \to \rrrr$.  Now we choose a parametrization 
of $\cL_2$ as follows:
fix  a $\pi_1(M)$ invariant  monotone graph $\zeta$ with
corresponding homeomorphism  $f: \cL_1 \to \cL_2$.
Notice that $f$ commutes with every deck translation $\gamma$.
Now define  the parametrization 
of $\cL_2$ as $h \circ f^{-1}$.

This gives a parametrization of the phase space as $\rrrr^2$,
given by $\lambda((x,y))  =  (h(x),  h(f^{-1}(y))$.
Now   use the fact that for any $\gamma$  in $\pi_1(M)$,
then $\gamma \circ f = f \circ \gamma$.
This implies that if $\gamma$ acts on $\cL_1$ as
a translation by $r$ (in the $\rrrr$ parametrization above),
then $\gamma$ also acts on  $\cL_2$ as a translation by $r$.

Consider now the map $\xi$ from $\mt$  to the reals given by:
if the coordinates of $\cD(p)$ are $(a,b)$ (in the $\rrrr^2$
parametrization above) let $\xi(p) = b - a$.
This  map  is continuous. Also if  $\gamma$ is in $\pi_1(M)$
and $p$ in $\mt$, with $\cD(p) = (x,y)$, then
$\xi(p) = h(f^{-1}(y)) - h(x)$.
Then $\gamma(p) = (\gamma(x), \gamma(y))$.
Also  $h(\gamma(x)) = h(x) + r$, and likewise 
$h(\gamma(y)) =  h(f^{-1}(y)) + r$, as we explained above.
Hence $\xi(\gamma(p)) = \xi(p)$.

Hence $\xi$ is  a $\pi_1(M)$ invariant map from $\mt$ to
$\rrrr$, which then induces a continuous map  from $M$ to
$\rrrr$. This implies that $\xi$ is bounded
and this is a contradiction as  follows:
fix a $x$ in  $\wcF_1$ and vary $y$ in $\wcF_2$ going to
infinity in either direction. This shows  that  $\xi$
is not bounded, contradiction.
We are using here that image  $\cD$ is  all  of $\cL_1 \times \cL_2$
to be able to fix $x$ and vary $y$   arbitrarily.
This contradiction proves  the lemma.
\end{proof}

\section{Group invariant collections of monotone graphs}\label{sec-groupinv}

In this section, as in the previous section, we assume that $\cF_1$ and $\cF_2$ are two minimal $\rrrr$-covered foliations, 
such that the actions of $\pi_1(M)$ on the leaf spaces $\cL_1$, $\cL_2$ are conjugated. But in this section we also assume the following:
\ 1) $M$ is a (closed) 3-manifold, \ 2)
Both $\cF_1$ and $\cF_2$ are Anosov foliations, and correspond to skew $\rrrr$-covered Anosov flows (equivalently, we assume that $\pi_1(M)$ is not solvable),
\ 3) $\cF_1$ and $\cF_2$ are and transversally orientable,
and $\cG$ is orientable (this assumption will be discussed in  \S~\ref{s.lastone}). 

According to Proposition \ref{pro:notdiscrete}, the union 
of $\pi_1(M)$-invariant monotone graphs is a closed subset of $\cL_1 \times \cL_2$, that we still slightly abusively denote by $\cC$,  satisfying the following:
\begin{enumerate}
\item $\cC$ is discrete: any compact set in the phase space
intersects at most finitely many elements in $\cC$, 
\item  $\cG$ is invariant under the diagonal action of $\pi_1(M)$.
\end{enumerate}

By $\pi_1(M)$ invariant we mean that for every $\gamma$ in
$\pi_1(M)$ and $\eta$ in $\cC$,  then
$\gamma(\eta)$ is in $\cC$, even though it may be a different
element of $\cC$.
In particular the conditions imply that $\cC$ is at most 
a countable union of monotone graphs.

\begin{proposition} \label{prop.closed}
Let $V$ be the projection to $M$ of a connected component
of $\cD^{-1}(\cC)$. Then $V$ is a closed embedded surface tangent to the foliation $\cG$ (in particular, it is a torus by our orientability assumptions), and contains
a closed  leaf of $\cG$. 
\end{proposition}

\begin{proof}
Let $\widetilde V$ be a connected component of $\cD^{-1}(\cC)$ and 
let $p \in \widetilde V$.
Notice first that $\widetilde V$ is $\wcG$ saturated.
Let $V = \pi(\widetilde V)$ and let 
$\eta$ be the element of $\cC$ containing $\cD(p)$.
Since $\cD$ is a submersion we can and choose
a compact transversal disk $D$ to $\wcG$,
having $p$ in its interior, so that $D$ projects homeomorphically
into its image by $\cD$, and the projection is a rectangle
in $\cL_1 \times \cL_2$ with two corners points in $\eta$,
having $\cD(p)$ in its interior.

We claim that there is neighborhood of $p$ so that
the only component of $\cD^{-1}(\cC)$ intersecting it is $\widetilde V$.
Choose $D$ small enough so that $\cD(D)$ only intersects the
element $\eta$ of $\cC$.
The image $\cD(D)$ contains a segment in $\eta$ with $\cD(p)$
in the interior, therefore $\widetilde V$ contains an
interval of $\wcG$ leaves with the leaf through $p$ in the interior.
Let now $p_n$ a sequence in $\cD^{-1}(\cC)$ converging to $p$.
For $n$ big then $\wcG(p_n)$ $-$ the leaf of $\wcG$ through $p_n$ $-$ 
intersects $D$. The previous paragraph shows that these leaves
of $\wcG$ are contained in $\widetilde V$. 
This proves the claim.


The above also shows that  $\cD^{-1}(\cC)$ is a discrete union of surfaces.
In addition it is 
$\pi_1(M)$ invariant, because $\cC$ is.
It follows that $\pi(\cD^{-1}(\cC))$ is a union of compact surfaces in $M$,
each component of which is foliated by $\cG$.
We stress that a priori some of the surfaces in $\cD^{-1}(\cC)$
may not be planes $-$ this in fact happens in the 
example from \cite{Ma-Ts}.


The only statement left to prove is that $V$ contains a closed 
leaf of $\cG$.
Since $V$ has a one dimensional foliation, by the orientability assumption,
it is a torus. 

Note first that if $V$ has a Reeb annulus, then, it must have a closed leaf, so we can assume without loss of generality that there is a global $S^1$-incompressible section to the foliation $\cG$ in $V$.


This section is a transversal to $\cG$ in $V$ and lifts to a $1$-dimensional line in $\widetilde{V}$, 
intersecting every leaf of the restriction of $\wcG$ to $\widetilde{V}$. The restriction
of $\cD$ to this transversal is a local homeomorphism into the graph $\eta$.
It follows that the restriction of $\cD$ to it is injective.  
Moreover, since no transversal to $\cG$ in $V$ cannot be nulhomotopic in $M$ (because it is transverse to $\cF_1$ and $\cF_2$) we deduce that $V$ is also incompressible. This implies that the stabilizer $\pi_1(V) \subset \pi_1(M)$ of $\widetilde V$ is isomorphic to $\zzzz^2.$  

Assume that $\cD(\widetilde  V)$ is not equal to 
the entire graph $\eta$. Then, its projection into the first component $\cL_1$ is not onto. 
This projection is either an interval $(a, b)$, or a half line $(a, +\infty)$ or $(-\infty,b)$ of $\cL_1.$ Every
extremity $a$ (or $b$) of this interval is then fixed by $\pi_1(V) \cong \zzzz^2.$
Since $\cF_1$ has no closed leaves (it is minimal) this is impossible. 

Therefore the image is all of $\cL_1$. Now we have a $\mathbb{Z}^2$
group acting on $\cL_1$. Since the flow associated with $\cL_1$
is a  skew $\rrrr$-covered Anosov flow, it follows
that $\mathbb{Z}^2$ does not act freely on $\cL_1$ (see \cite{Fe1}),
and therefore  there is $\gamma \in \pi_1(V) \setminus \{\mathrm{id}\}$ so that
$\gamma$ fixes  a point $x$ in $\cL_1$. It  follows that
$\gamma$ fixes $(x,y)$ where $(x,y)$ is the intersection
of $\{ x \} \times \cL_2$ with $\eta$. If $\widetilde \theta$ is the
unique leaf of $\wcG$ in  $\widetilde V$ which projects
to $(x, y)$ by $\cD$, it follows that  $\gamma(\widetilde  
\theta) = \widetilde \theta$. Therefore  $\pi(\widetilde \theta)$
is a closed curve of $\cG$ in $V$. 

This finishes the proof of Proposition \ref{prop.closed}.
\end{proof}

Note that we have used the fact that $\cF_1$ comes from a skew $\rrrr$-covered Anosov flow only to ensure that $\zzzz^2$ cannot act freely in the leaf space, but this is something that holds in general for instance if $M$ is hyperbolic.

The following lemma is extremely useful, and
it is one of the few results where 
the hypothesis that $\cF_1$ has no Reeb annulus is used explicitly. The argument appeared implicitly in the previous proof and we now make it explicit in a more general context. 

\begin{lemma}\label{le.uniqueperiodic}
Assume that $\cG$ has no Reeb annulus in $\cF_1.$
Let $\widetilde \theta$ and $\widetilde \theta'$ be two orbits of $\wcG$ 
so that $\theta := \pi(\widetilde \theta)$ and $\theta' := \pi(\widetilde \theta')$
are closed leaves of $\cG$.
Then if $\cD(\widetilde \theta) = \cD(\widetilde  \theta')$
it follows that $\widetilde \theta = \widetilde \theta'$.
\end{lemma}

\begin{proof}
Suppose by way of contradiction that  $\widetilde \theta$ is
not $\widetilde \theta'$. 
    Since $\cD(\widetilde \theta) = \cD(\widetilde \theta')$
it is equal to $(x,y) \in \cL_1 \times \cL_2$, a point in the image of  $\cD$.     
    Hence $\widetilde \theta$ and $\widetilde \theta'$ belong to the same leaf $x$ of $\wcF_1$. 
    
    Let $F$ be the projection of $x$ in $M$. Since $\cF_1$ is an Anosov foliation, it follows that $F$ is an annular leaf of $\cF_1$ containing the closed
leaves $\theta$ and $\theta'$ of $\cG$. 
    
Denote by $\gamma$ a generator of the (cyclic) fundamental group of $F$.
It follows  that $\gamma(\widetilde  \theta) = \widetilde  \theta$.
In particular the projection  of $\widetilde \theta'$  to  $F$ (that
is the closed curve  $\theta'$)  is  not equal  to $\theta$ $-$ this
is where the $\widetilde \theta \not = \widetilde \theta'$ is used.
Let $A$ be the annulus in $F$ bounded by $\theta$ and $\theta'.$ Since by hypothesis $F$ contains no Reeb annulus of $\cG$ there is a curve $c$ in $F$ transverse to $\cG$ and intersecting $\theta$ and $\theta'$. It lifts in $x$ as a curve $\tilde c$ transverse to $\wcG$ and intersecting $\widetilde \theta$ and $\widetilde \theta'.$ But this curve is transverse to $\wcF_2$. Hence the restriction of $\mu_2$ to $\tilde c$ is a locally injective map between two $1$-dimensional manifolds, hence injective. Therefore $\cD(\widetilde \theta) \neq \cD(\widetilde \theta), $ which is a contradiction and completes the proof. 
\end{proof}

Obviously the same result holds replacing $\cF_1$ by $\cF_2$. 

\begin{remark} The reader may notice that a lot of the difficulty
revolves around the fact that the induced map $\cD_{\cG}$ from the
leaf space $\cL_{\cG}$ of  $\wcG$ to the phase space may  not be
injective. Indeed in the examples of Matsumoto and  Tsuboi
\cite{Ma-Ts} this map is not injective. This is one of
the main complications  involved  in our study.
\end{remark}


The next lemma goes further in the description of the foliation $\cG$ restricted to the tori $V$ saturated by $\cG$ leaves that is found in Proposition \ref{prop.closed}. 

\begin{lemma}\label{lem-hsdff}
Assume that $\cG$ does not have a Reeb annulus in a leaf of $\cF_1$.
    The restriction of $\cG$ to $V$ is $\rrrr$-covered.
\end{lemma}

\begin{proof}
    Assume not. Then, this foliation contains a Reeb annulus $R$ bounded by two different closed leaves $\theta$ and $\theta'$ of $\cG$. They lift to  leaves $\widetilde \theta$ and $\widetilde \theta'$ of $\wcG$ in $\widetilde V$, invariant by some non-trivial element $\gamma$ of $\pi_1(M)$, and are not separated. It implies that $\widetilde \theta$ and $\widetilde \theta'$ have the same image under $\cD$. It contradicts Lemma \ref{le.uniqueperiodic} completing the proof. 
\end{proof}

Finally we gather all the properties of the tori in the next proposition. 

\begin{proposition}\label{prop-toriprops}
Assume that $\cG$ does not have Reeb annuli in leaves of $\cF_1$.
        For any invariant monotone graph $\eta$ intersecting the image of $\cD$, and every connected component $\widetilde V$ of $\cD^{-1}(\eta)$, the restriction of $\cD$ to $\widetilde V$ is a locally trivial fibration over the entire $\eta$, with fibers the leaves of $\wcG$ contained in $\widetilde V.$ Moreover, the stabilizer of $\widetilde V$ is a subgroup $H$ of $\pi_1(M)$ isomorphic to $\zzzz^2$. Elements of $H$ admitting fixed points in $\eta$ form a cyclic subgroup $H_0$ generated by an element $\gamma_0$. 
\end{proposition}

\begin{proof}
It is just an application of arguments that we have already applied several times. The restriction of $\cD$ to $\widetilde V$ induces a local submersion from the leaf space of $\wcG$ into $\eta$.  As in the proof of Proposition \ref{prop.closed} we infer that this leaf space is homeomorphic to the real line and not the circle, and  that the torus $V$ is incompressible. Hence its fundamental group $H$ is isomorphic to $\zzzz^2.$ The rest of the proposition follows.    
\end{proof}

\begin{remark}[Hyperbolic 3-manifolds]\label{rem-hyperbolic}
We note here that the previous proposition implies that in the case where $M$ is a hyperbolic 3-manifold, the developing map $\cD$ cannot intersect any invariant graph. Thus, Proposition \ref{prop-nointersectionAnosov} below together with Proposition \ref{prop-toriprops} completes the proof of Theorem \ref{thm.three} for the case where $M$ is a hyperbolic 3-manifold. To complete the proof in the general case, we will need further arguments developed in \S~\ref{sec-nonseparatedleaves} and \S~\ref{sec.intersect}. 

\end{remark}

\section{Disjointness with invariant graphs implies Anosov foliation}
\label{sec.nointersect}

In this and the next  two sections
we intend to prove Theorem \ref{thm.three} under some orientability assumptions.
Hence we assume that $M$ has dimension $3,$ that $\cF_1, \cF_2$ are each isotopic to the weak stable and weak unstable foliation of a skew $\rrrr$-covered Anosov flow $\Phi$ (cf. section \ref{sub.defanosov}). In particular, the actions of $\pi_1(M)$ on $\cL_1$ and $\cL_2$ are conjugate to each other and the foliations isotopic. We further assume that $\cF_1$ and $\cF_2$ are orientable and transversally oriented. 

We will moreover assume that $\cF_1$ does not admit a Reeb annulus for $\cG$, and the goal is to prove that $\cG$ is topologically equivalent to the orbit foliation of an Anosov flow. Note that if $\cF_1$ has a Reeb annulus, so does  $\cF_2$. Indeed, if $A$ is a Reeb annulus in a leaf of $\cF_1$ and $L$ is the lift of such a leaf, then, the boundaries of lift of the Reeb annulus must be contained in the same leaf $E \in \wcF_2$ because $\cF_2$ is $\rrrr$-covered. Thus, $E$ projects to an annulus leaf containing two homotopic circles of $\cG$. There is no transversal to $\cG$ in $\pi(E)$ between these homotopic circles,
because the circles lift to the same leaf of $\wcF_1$. This implies that $\pi(E)$ also contains a Reeb annulus.

Recall from \S~\ref{sec:invariant} that the two actions are conjugated if and only if there exists an invariant graph of a homeomorphism $h: \cL_1 \to \cL_2$ conjugating the actions of $\pi_1(M)$ on each leaf space. 


In this section we deal with the case where the image of $\cD$ does not intersect any $\pi_1(M)$ invariant monotone graph.  In section 
\S~\ref{sec.intersect} we will show that the other case, i.e. when
the image of $\cD$ intersects some invariant monotone graph, leads to a contradiction (more precisely, we show that if $\cD$ intersects an invariant monotone graph, then $\cF_1$ has a Reeb annulus).
In the next section we prove a couple of technical results that
are needed later on.

\begin{remark}
We shall use the assumption that $\cF_1$ (and thus $\cF_2$) does not have Reeb annuli. In this section, this is used in Lemma \ref{le.periodicisperiodic} below. Note that in principle, it could be that this is not needed, namely, it is conceivable that the non intersection of the image of $\cD$ with the invariant graphs is enough to show that the intersected foliation is an Anosov foliation. This holds in unit tangent bundles (see \cite{FP2}).
\end{remark}

Remember that $\cF_1, \cF_2$ are transversally oriented
skew Anosov foliations, in particular $\cF_i$ are $\rrrr$-covered. For each $i=1,2$ there is a homeomorphism $\tau_i: \cL_i \to \cL_i$
commuting with the action of $\pi_1(M)$, and without fixed points: for every $x$ in $\cL_i$ we have $x <_i \tau_i(x).$ 
In addition the leaves or $\cF_i$ are only planes and  annuli.

Since $\cF_1$ and $\cF_2$ are conjugated to the weak  stable (or weak  unstable foliation)
of the same skew $\rrrr$-covered Anosov flow $\Phi$  by assumption,  there are homeomorphisms $g_1, g_2$  which are
isotopic  to each other, and  $g_1, g_2$  send  $\cF_1,  \cF_2$ respectively
to the weak stable foliation and the weak unstable foliations of $\Phi$. Note that under the transverse orientability assumption, the weak stable and weak unstable foliations of $\Phi$ are isotopic in the sense of \S~\ref{sub.homotopic}, so, one can adopt the point of view that $\cF_1$ is isotopic to the weak  stable foliation $\cF^s$ of $\Phi$, and $\cF_2$ is isotopic  to the weak  unstable foliation $\cF^u.$ We will thus identify $\cL_1$ with $\cL^s$ (the leaf space of $\widetilde{\cF^s}$) 
and $\cL_2$ with $\cL^u$ (the leaf space of $\widetilde{\cF^u}$) as
follows:

Since the actions of $\pi_1(M)$ on $\cL_1$ and $\cL_2$
are  conjugate, there is a $\pi_1(M)$ invariant monotone graph.
Under the hypothesis  of this section, this implies that
the image of $\cD$ is not all of the phase space, and
hence either
$\alpha$ or $\beta$ is finite. 
Then because $\cF_1$ is isotopic to $\cF^s$ (of $\Phi$) and
$\cF_2$ is isotopic to $\cF^u$ (of $\Phi$), it follows that both
$\alpha$ and $\beta$ are finite (recall Definition~\ref{def.alphabeta}).
In  addition  $\alpha$ is a $\pi_1(M)$ equivariant homeomorphism
from $\cL_1$ to $\cL_2$: 
for any $L$ in $\cL_1$ and  $\gamma$ in $\pi_1(M)$ then
$\gamma \circ \alpha(L) = \alpha \circ \gamma(L)$.

%


Let $\nu_1: \cL_1 \to \cL^s$ be the map that sends a leaf 
of $\cF_1$ to itself (thought of as a leaf of $\cL^s$).

Notice that all foliations are in $M$, so any $\gamma$ in $\pi_1(M)$
acts on the objects in $\mt$, in the phase space $\cL_1 \times
\cL_2$ or $\cL^s \times \cL^u$.

Now since $\Phi$ is a skew  Anosov flow, it also has
associated increasing homeomorphisms $\alpha_0, \beta_0: \cL^s
\to \cL^u$. We define a map $\nu_2: \cL_2 \to \cL^u$ by

$$\nu_2(E) = \alpha_0 \circ \nu_1 \circ \alpha^{-1}(E)$$

\noindent
Finally define $\nu: \cL_1 \times \cL_2 \to \cL^s  \times \cL^u$
by 

$$\nu(x,y) \ = \ (\nu_1(x), \nu_2(y))$$

\noindent
Then $\nu$ is a homeomorphism, which is $\pi_1(M)$ equivariant.
Therefore  any $\pi_1(M)$ invariant  graph in $\cL_1 \times \cL_2$
is sent to an invariant graph in 
$\cL^s \times \cL^u$.

\begin{remark}\label{rem-omegakl} According to Proposition \ref{pro:alphaLL}, the invariant monotone graphs in $\cL^s \times \cL^u$ (associated
with the Anosov flow  $\Phi$) are precisely the graphs $\zeta_k$ of the maps $\beta_0 \circ \tau_1^k$, where $k$ is an integer (again
we stress that the maps $\alpha_0, \beta_0$ are different from
$\alpha, \beta$. $\alpha_0$, $\beta_0$ indicate how the weak  stable and unstable leaves of the Anosov flow $\Phi$ intersect each
other, and the maps $\alpha$, $\beta$ do the same but for the new foliations $\cF_1$, $\cF_2$). Let $\Omega_k$ be the region between the graphs of two successive invariant monotone graphs $\zeta_k$ and $\zeta_{k+1}$. Then
$\Omega_k$  is a copy of the orbit space of the Anosov flow $\Phi$. 
\end{remark}

The graphs of $\alpha$ and $\beta$ are $\pi_1(M)$ invariant monotone
graphs. 
It follows that the  image of $\nu \circ \cD$ is the region between two invariant monotone graphs $\zeta_p$ and $\zeta_q$ in $\cL^s \times \cL^u$ which are consecutive (i.e. $q=p+1$) by the remark above and the fact that we have assumed that the image of $\cD$ does not intersect any invariant graph. 
Without  loss of generality we can assume $p=0$, $q=1$, so that the image of $\cD \circ  \nu$ is the region $\Omega_0$. Recall that this region is canonically identified with the orbit space of $\Phi$.

\begin{lemma}\label{le.periodicisperiodic}
    Let $(x,y)$ be an element of $\Omega_0$ fixed by some non trivial element $\gamma$ of $\pi_1(M).$ Then $\cD^{-1}(\nu^{-1}((x,y))$ is connected: it is a single leaf of $\wcG$ preserved by $\gamma$. 
    \end{lemma}

\begin{proof}
    According to Lemma \ref{le:periodicdiscrete} the $\pi_1(M)$-orbit of $(x,y)$ is discrete in $\Omega_0$. 
Therefore the $\pi_1(M)$-orbit  of $\nu^{-1}(x,y)$ is
also discrete in the image of $\cD$. We want to prove a sort of converse to this: the preimage by $\cD$ of $\nu^{-1}(x,y)$ is a discrete union of curves, in particular, projects to closed curves in $M$. For this we will use the properties of $\cD$ and $\cD_{\cG}$. 

In fact, since $\cD$ is locally a trivial fibration with fibers being leaves of $\wcG$ (i.e. the map $\cD_{\cG}$ is a local homeomorphism, recall \S~\ref{ss.developing}), it follows that the preimage by $\cD_{\cG}$ of a discrete set is also discrete in $\cL_{\cG}$. In particular, if we consider the preimage of the $\pi_1(M)$-orbit of $\nu^{-1}(x,y)$ by $\cD_{\cG}$ is discrete in $\cL_{\cG}$ and thus, the $\pi_1(M)$-orbit of $\cD^{-1}(\nu^{-1}(x,y))$ is a closed set in $\mt$. This implies that any connected component of $\cD^{-1}(\nu^{-1}(x,y))$ projects to a closed set in $M$ and thus is a closed curve. In conclusion, each connected component of $\cD^{-1}(\nu^{-1}(x,y))$ is the lift of a closed orbit of $\cG$. 

Suppose that $c_1, c_2$ are different connected components of $\cD^{-1}(\nu^{-1}(x,y))$.
Let $L$ be the leaf of $\wcF_1$ containing both $c_1, c_2$, and these curves project to closed curves $\pi(c_1), \pi(c_2)$  in $M$. Lemma \ref{le.uniqueperiodic} shows that this is impossible and thus we get that there is a unique connected component of $\cD^{-1}(\nu^{-1}(x,y))$ as desired.    
\end{proof}

We now pass the information from periodic orbits to all orbits to be able to apply Corollary \ref{cor.F1F2}. 

\begin{proposition}\label{prop-hsdff}
    The orbit space $\cL_{\cG}$ is Hausdorff, and the map $\cD_{\cG}: \cL_{\cG} \to \Omega_0$ is a homeomorphism.
\end{proposition}

\begin{proof}
In order to apply Corollary \ref{cor.F1F2} it is enough to show that every pair of leaves $L \in \cL_1$ and $E \in \cL_2$ intersect in a unique connected component. We will show that if there is a pair of leaves $L_0 \in \wcF_1$ and $E_0 \in \wcF_2$ so that the intersection $L_0\cap E_0$ has at least two connected components, then, there is an open set in $\cL_1 \times \cL_2$ with the same property. Since the set of points $(x,y) \in \Omega_0$ fixed by some deck transformation are dense in $\Omega_0$, showing this completes the proof because one can apply Lemma \ref{le.periodicisperiodic} to get a contradiction.

To prove our claim, 
we consider curves $\ell_1, \ell_2 \in \wcG$ which are different connected components of $L_0 \cap E_0$. We can assume that they are different boundary components of the same connected component of the complement of $L_0 \cap E_0$ in $E_0.$
We first choose a transversal $\tau_1: (-\eps,\eps) \to L_0$  to $\cG_{L_0}$ so that $\tau_1(0) \in \ell_1$. This gives a parametrization $E_s$ of leaves of $\wcF_2$ with $s \in (-\eps,\eps)$. Note that up to reparametrization we can assume that for $s<0$ we have that $\ell_1$ separates $\tau_1(s)$ from $\ell_2$. This implies that for small $s<0$ the leaf $E_s$ intersects $L_0$ in at least two components, the one intersecting $\tau_1(s)$ and the one close to $\ell_2$. So, by changing $E_0$ with some $E_{s_0}$ with small $s_0<0$ we can assume without loss of generality that $E_s$ intersects $L_0$ in two connected components $\ell_1^s$ and $\ell_2^s$ for all $s \in (-\eps,\eps)$. 

According to our choice of  $\ell_1, \ell_2 $ there is a curve $\gamma: [0,1] \to E_0$ which intersects $L_0$ only in the extremes, it is transverse to $\wcF_1$ for $t \in [0,\eps) \cup (1-\eps,1]$ and so that $\gamma(0)$ and $\gamma(1)$ belong to different connected components of $L_0 \cap E_0$. We can assume without loss of generality (up to changing the components) that we have have $\gamma(0) = \tau_1(0) \in \ell_1$ and $\gamma(1) \in \ell_2$. We consider the curves $\ell_1^s$ and $\ell_2^s$ to be the connected components of $E_s \cap L_0$ close to $\ell_1$ and $\ell_2$ respectively as above. Choose $\tau_2^0: (-\eps,\eps) \to E_0$ a transversal to $\wcF_1$ so that $\tau_2^0$ coincides with $\gamma$ for small $t \in [0,\eps)$. We can also assume that $\gamma$ is transverse to $\wcF_1$ for $t \in (1-\eps,1]$. This way, it follows that up to reparametrization, we can assume that $\gamma(t)$ and $\gamma(1-t)$ belong to the same leaf of $\wcF_1$. 
Indeed, since $\gamma$ does intersect $L_0$ only at its extremities, it has to come back in $\cL_1$ to $L_0$ from the same side it left it at the beginning.
In particular, if we consider a transversal $\tau_3^0: (-\eps,\eps) \to E_0$ to $\wcF_1$ so that $\tau_3^0(t)= \gamma(1-t)$ we get that $\ell_1$ and $\ell_2$ separate in $E_0$ the points $\tau_2^0(u)$ and $\tau_3^0(u')$ for $u,u' <0$.

By choosing a foliated chart around $\tau_2^0(0)$ and $\tau_3^0(0)$ and maybe reducing $\eps$ we can find, for $s \in (-\eps, \eps)$ a continuous family of curves  $\tau^s_2: (-\eps, \eps) \to E_s$ and $\tau^s_3: (-\eps,\eps) \to E_s$ transverse to $\wcF_1$ so that for $t\in (-\eps, \eps)$ we have that $\tau^s_2(t), \tau_3^s(t)$ and $\tau_2^0(t)$ belong to the same leaf of $\wcF_1$. Also, we can push the curve $\gamma$ to a family of curves $\gamma^s: [0,1] \to E_s$ which coincide with $\tau_2^s$ and $\tau_3^s$ in the extremes, and as above intersect $L_0$ only at the endpoints (which lie in different connected components of the intersection of $L_0 \cap E_s$, namely $\ell_1^s$ and $\ell_2^s$). Note that this can be done because at the extremes we have chosen $\gamma^s$ to coincide with $\tau_2^s$ and $\tau_3^s$ which can be constructed explicitly in the foliated chart, and then, $\gamma|_{[\eps, 1-\eps]}$ lies at positive distance to $L_0$ so for small $s$ one can push the curve to nearby $E_s$ leaves (one chooses a foliated chart of $\wcF_2$ containing $\gamma|_{[\eps, 1-\eps]}$ and defines $\gamma^s$ in these coordinates). 

By this construction it follows that $\tau_2^s(t)$ is separated from $\tau_3^s(t)$ in $E^s$ by $\ell_1^s$ (and $\ell_2^s$) for all small $t<0$. In particular for every $s,t$ small so that $t<0$ we have that $\ell_1^s$ separates $\tau_2^s(t)$ from $\tau_3^s(t)$ in $E_s$. We conclude that for small $t<0$ the intersection $L_t \cap E_s$ has two different connected components, one containing $\tau_2^s(t)$ and the other $\tau_3^s(t)$. 

We have obtained from the existence of a pair of leaves intersecting in two distinct connected components an open set of such leaves. As we explained at the beginning of the proof, this produces a contradiction with Lemma \ref{le.periodicisperiodic} and thus completes the proof. 
\end{proof}

The following proposition shows that in order to prove Theorem \ref{thm.three} it is enough to show that the image of $\cD$ does not intersect any invariant monotone graph. We could apply the main result of \cite{FP3} to obtain this due to the previous proposition, but we can give a shorter argument in this case taking advantage of the properties of the map $\cD$.

\begin{proposition}\label{prop-nointersectionAnosov}
 Let $\cF_1$ and $\cF_2$ be transverse foliations so that there is a transversally orientable skew-$\rrrr$-covered Anosov flow $\Phi$ with weak foliations $\cF^s, \cF^u$ and so that $\cF_1$ is isotopic to $\cF^s$ and $\cF_2$ isotopic to $\cF^u$. Assume that the image of $\cD$ does not intersect an invariant monotone graph, and that $\cF_1$ has no Reeb annulus. Then, $\cG = \cF_1 \cap \cF_2$ is topologically conjugate to $\Phi.$
\end{proposition}

\begin{proof}
    We have proved that $\nu \circ \cD_{\cG}$ is a $\pi_1(M)$-equivariant homeomorphism between the leaf space $\cL_{\cG}$ of $\cG$ and the orbit space $\Omega_0$ of the Anosov flow $\Phi.$ 

    Moreover, $(M, \Phi)$ and $(M, \cG)$ are both  classifying spaces of their holonomy groupo\"{\i}d (since all its leaves have contractible holonomy covering). Therefore, they are homotopically equivalent by \cite{Hae}. There is a continuous homotopy equivalence $h: M \to M$ mapping every orbit of $\Phi$ into leaves of $\cG$. 
    
   In addition this maps lifts to a map $\tilde h: \mt \to \mt$ mapping orbits of the lifted Anosov flow $\widetilde \Phi$ into leaves of $\wcG$, and inducing between the leaf spaces the inverse of the homeomorphism $\nu \circ \cD_{\cG}.$ It follows that $h$ has a ``transverse injectivity" property, and that it maps leaves of $\cF^s$ into leaves of $\cF$, and leaves of $\cF^u$ into leaves of $\cF_2$ (see e.g. \cite[Proposition 5.6]{BFP}).  

By our orientability assumptions, and the hypothesis that
leaves of $\cG$ are $C^1$, it follows that
leaves of $\cG$ are orbits of some flow $g^t$.

    There is a continuous map $u: \rrrr \times M \to M$ such that, for every real number $t$ and every $p$ in $M$ we have:
    $$h(\Phi^t(p)) = g^{u(t,p)}(h(p))$$

\begin{claim} 
    There exists $T > 0$ such that for every $p$ in $M$ we have $u(t,p) \neq 0$ for every $t>T$. 
\end{claim}

\begin{proof}
Let us prove this claim (a proof of a similar result has already been written in \cite{Ba1}). 

Assume by a way of contradiction that there is a sequence of points $p_n \in M$ so that $p_n \to p_\infty$ and a sequence $T_n \to +\infty$ such that $u(T_n, p_n) = 0$. We fix lifts $\tilde p_n \to \tilde p_\infty$ of the points in $\mt$. One can lift $h$ to a map $H$ at bounded distance from the identity that will continue to satisfy for every $t,x$: 

 $$H(\tilde \Phi^t(x)) = \tilde g^{u(t, \pi(x))}(H(x)), $$
 
 \noindent where $\pi: \mt \to M$ denotes the universal covering projection. 
  
   Note then that $H(\tilde \Phi^{T_n}(\tilde p_n)) = H(\tilde p_n) \to H(\tilde p_\infty)$. In particular, since $H$ is a bounded distance from identity, both $\tilde p_n$ and $\tilde  \Phi^{T_n}(\tilde p_n)$ belong to a uniform neighborhood of $\tilde p_\infty$ and thus we can assume, up to subsequence that $\tilde  \Phi^{T_n}(\tilde p_n)$ converges to some point $\tilde q_\infty$. We know that since the orbits of $\tilde \Phi$ are properly embedded we have that $\tilde q_\infty \neq \tilde p_\infty$ but also cannot belong to the same orbit. This contradicts the fact that $\cL_{\cG}$ is Hausdorff as proved in the previous proposition and gives a contradiction showing the claim. 
\end{proof}

The claim is enough to ensure that $h$ can be modified by an averaging method along orbits of $\Phi$ so that it becomes a homeomorphism, as explained for instance in  \cite[\S 5]{Ma-Ts}. This finishes the proof of the proposition. 
\end{proof}

\section{Non separated leaves and expansion of holonomy}\label{sec-nonseparatedleaves} 

The purpose of this section is to prove a couple of technical 
results, which will be used in the next section. This section is independent from the rest of the article, and it will have weaker assumptions on the foliations as we believe that the result may be interesting by itself and possibly useful in other contexts. After this is proved, we will see that for Anosov foliations, expansion in holonomy in one side, implies expansion on both sides, this will be useful also in the next section. 

\subsection{Simultaneous non-separation} 

Here the context is two transverse foliation $\cF_1$ and $\cF_2$ in a closed 3-manifold $M$, but we do not assume the foliations  to be $\rrrr$-covered. The result refines and gives a better understanding of the situation of two leaves in $\mt$ intersecting in more than one connected component (compare to \cite[\S 8]{FP2}). Recall that if we denote by $\cG$ the one dimensional intersection foliation $\cG = \cF_1 \cap \cF_2$ and we denote the lifts to the universal cover by $\wcG, \wcF_1, \wcF_2$, then, for a given leaf $L \in \wcF_i$ we denote by $\wcG_L$ the foliation $\wcG$ restricted to $L$. 

In particular in this section we do not assume that $\cF_1, \cF_2$
are transversely orientable, nor that $\cG$ is orientable.

\begin{theorem} \label{nonsep}
Suppose that $\cF_1, \cF_2$ are transverse two dimensional foliations in a closed 3-manifold $M$ and 
let $\cG = \cF_1 \cap \cF_2$. Suppose there are leaves $L \in \wcF_1$ and $E \in \wcF_2$ which intersect in at least two distinct leaves $c_1, c_2$ of $\wcG$. Assume moreover that $c_1, c_2$ are non-separated in $\wcG_{L}$. Then there is $L' \in \wcF_1$ so that the intersection $E \cap L'$ contains leaves $\ell_1, \ell_2 \in \wcG$ which are non separated in both $\wcG_{E}$ and $\wcG_{L'}$. 
\end{theorem}


\begin{proof}

We first look at $\wcG_E$. We consider the \emph{pseudo-interval} from $c_1$ to $c_2$ in
the leaf space of $\wcG_E$, this is the following:

\begin{equation}\label{eq:pathleaf}
[c_1,c_2] \ \ := \ \ \bigcup \{ \ell \in \wcG_E \ : \  \ell \ {\rm separates} \ c_1 \ {\rm from} \ c_2  \ {\rm in } \ E \}  \cup \{ c_1 \} \cup \{ c_2 \}.
\end{equation}

\noindent
The pseudo-interval $[c_1,c_2]$ consists of a finite union
$\cup_{i=1}^n  [x_i,y_i]$,  where each $[x_i,y_i]$ is a closed interval,
$x_1 = c_1$, $x_2 = c_2$ and for any $1 \leq i \leq n-1$, the leaf
$y_i$ is non separated from $x_{i+1}$ in $\wcG_E$ (see \cite[Lemma 2.2]{Barb-Fe1} or \cite{Ba3}). 
The integer $n$ is the \emph{length} of the pseudo-interval. 
Note that since $c_1, c_2 \in L$ they cannot intersect a common transversal, thus, $n \geq 2$.
In addition $[c_1,c_2]$ is always compact. (See figure~\ref{fig.nhsdf}.) This pseudo-interval may intersect $L$ in leaves which are not $c_1, c_2$. This makes the visualization in $3$-dimensions a little harder, but
does not really affect the arguments.

Note here that if $c_1=x_1=y_1$ and $x_2=y_2=c_2$ then we already have that $c_1$ and $c_2$ are non separated in $\wcG_E$ and there is nothing to prove.

Our strategy will be to prove that if $n>2$, then we can find another pair $\{ c'_1, c'_2\}$ of leaves of $\wcG_E$ satisfying exactly the same hypothesis as $\{ c_1, c_2 \}$, but with length $<n.$ Therefore, 
iterating the process, we are led to the case $n=2$, and a similar argument will show that the pair $\{y_1, x_2\}$ satisfies the properties required by Theorem \ref{nonsep}.

We fix  points $p$ in $c_1$ and $q$ in $c_2$. Since $c_1,c_2$ are non separated in $L$, one can choose two transversals to $\cG_L$ in $L$ denoted by
$\tau_1,\tau_2: (-\eps, \eps) \to L$ to $\wcG_L$ so that $\tau_1(0)=p$, $\tau_2(0)=q$ and such that $\tau_1(t)$ and $\tau_2(t)$ belong to the same leaf $E_t$ of $\wcF_2$ if $t > 0$. Moreover, the non separation implies that, up to reparametrization, we can assume that for any given $t >0$ we have that the curve of $\wcG_{E_t}$ through $\tau_1(t)$ contains a compact segment $I_t$ joining $\tau_1(t)$ with $\tau_2(t)$. By non-separation, the intervals $I_t$ converge to rays $r_1, r_2$ of $c_1$ and $c_2$ respectively (they may also converge to other non-separated leaves of $\wcG_L$ different from $c_1$ and $c_2$). These rays $r_1,r_2$, starting at $p,q$ respectively will be called the non-separated rays. The endpoints of  $I_t$ are $p_t =\tau_1(t)$ and $q_t= \tau_2(t)$, notice that $\{ E_t \}$ and $E$ are pairwise disjoint as they intersect a tranversal to $\wcF_2$.

We can choose a continuous path $\curve_0: [0,1] \to E$ so that $\curve_0(0)=p$ and $\curve_0(1) = q$ which is  \emph{efficient}, that is, so that for $\ell$ 
in $[c_1,c_2]$, then $\curve_0$ intersects $\ell$ in a single point and
to go from $y_i$ to $x_{i+1}$ the path has finitely many tangency points 
with $\wcG_E$ and intersects leaves in uniformly finitely many points (this can be done by a standard general position argument, note that it is possible to construct paths so that for going from $y_i$ to $x_{i+1}$ there is a unique tangency and intersects each leaf in at most two points).

One crucial feature of this path $\curve_0$ is that it does not intersect
$c_1 \cup c_2$ in the interior.
But $\curve_0$ can intersect $L$ in the interior as explained previously.

\begin{figure}[ht]
\begin{center}
\includegraphics[scale=0.60]{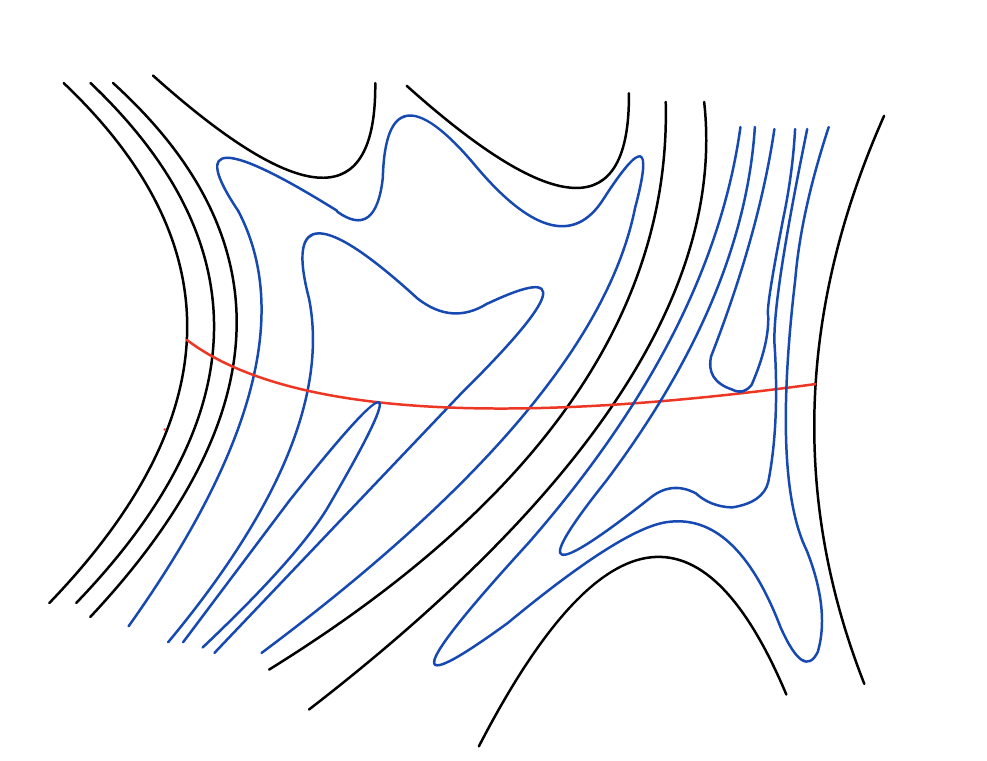}
\begin{picture}(0,0)
\put(-284,40){$c_1$}
\put(-40,190){$c_2$}
\put(-270,38){$y_1$}
\put(-220,21){$x_2$}
\put(-208,10){$y_2$}
\put(-188,110){$\curve_0$}
\end{picture}
\end{center}
\vspace{-0.5cm}
\caption{{\small The pseudo-interval between $c_1$ and $c_2$ consists of the curves which separate $c_1$ from $c_2$, in the picture, these are represented by the black curves intersected by the red curve (which represents $\curve_0$).}}\label{fig.nhsdf}
\end{figure}
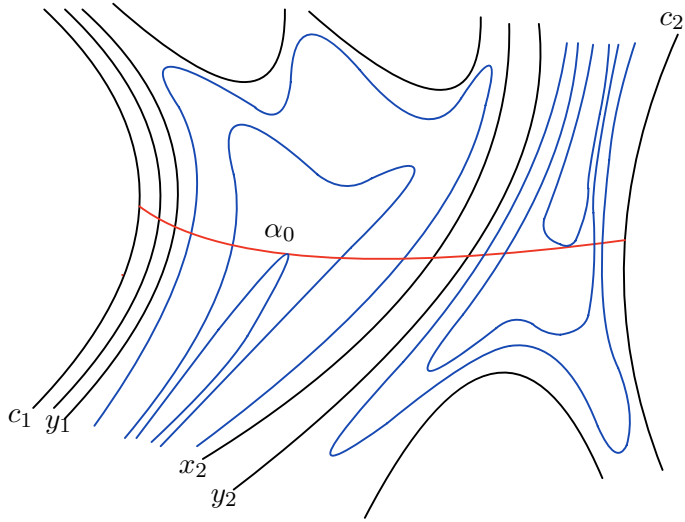

We can construct a continuous family of transversals to $\wcF_2$
in $\mt$, denoted by  $a_s: [0, \eps) \to \mt$ to $\wcF_2$ for $s \in [0,1]$ so that $a_0= \tau_1|_{[0,\eps)}$ and $a_1= \tau_2|_{[0,\eps)}$ with the properties that:

\begin{itemize}
\item $a_s(0) = \curve_0(s)$, 
\item $a_s(t) \in E_t$
\end{itemize}

Up to reducing $\eps$, this can be done by considering a vector field transverse to $\wcF_2$ in a neighborhood\footnote{Here the assumption that the foliations are tangent to a continuous plane field is useful.} of $\curve_0$ and tangent at the extremes to the transversals $\tau_1$ and $\tau_2$.

Let $\curve_t: [0,1] \to E_t$ be equal to $\curve_t(s) = a_s(t)$ which is going from $p_t$ to $q_t$, which from how we chose $\curve_0$ we know intersects $I_t$ only at the endpoints, at least for small $t$.

Let $\theta_t = \curve_t \cup I_t$. This is a simple closed
curve in $E_t$, which therefore bounds a unique disk $D_t$
inside $E_t$. As $t \rightarrow 0$, then $D_t$ limits to a region in
$E$ whose boundary contains the rays $r_1, r_2$ (and maybe
other leaves of $\wcG_E$ as well).
In particular the disks $D_t$ have diameter that is going to
infinity.

\begin{claim}\label{i+1}
For some $i$, the points $y_i, x_{i+1}$ are in the same leaf of $\wcF_1$.
\end{claim} 

\begin{proof}
Note that this claim is immediate if $\wcF_1$ is $\rrrr$-covered. We give the proof in the general case. 

Consider the set $\mathcal P$ of pairs of different leaves $(u, v)$ of $\wcG_E$ in $[c_1, c_2]$ such that $u$ and $v$ are contained in the same leaf of $\wcF_1$ (for example, $(c_1, c_2)$ since $c_1$ and $c_2$ are contained in $L$).
Notice that $[u,v] \subset [c_1, c_2]$.
Observe that $u$ and $v$ cannot lie in the same component $[x_i, y_i]$ since they would be intersected by a common transversal. 
We choose $(u,v)$ in $\cP$ so that $(u,v)$ is a minimal element in the 
following way:
if $u \in [x_i,y_i]$, and $v \in [x_j,y_j]$ with
$i < j$ (which can be achieved by switching $u, v$) then $j - i$ is minimized
amongst all elements in $\cP$.

Fix now a minimal element $(u_0,v_0) \in \mathcal{P}$. Let $i$ be the index of the component $[x_i, y_i]$ containing $u_0$, and $j>i$ the index of the component $[x_j, y_j]$ containing $v_0$.
Select the transverse orientations so that, as leaves of $\wcG_E$, $x_{i+1}$ lies in the component $y_i^-$ and $y_i$ in $x_{i+1}^-.$ Then, $u_0$ lies in the closure of $y_i^+$ and $v_0$ lies in the closure of $x_{i+1}^+.$ 

Let $L_i$ and $L_{i+1}$ the leaves of $\wcF_1$ containing 
$y_i$ and $x_{i+1}$, respectively. 
Assume $L_i \neq L_{i+1}$. Then they are not separated in $\cL_1$, and it follows that the closures of $L_i^+$ and $L_{i+1}^+$ are disjoint in $\mt.$ On the other hand,  the leaf $L_u$ containing $u_0$, if not equal to $L_i$, is contained in the connected component $L_i^+$ of $\mt \setminus L_i$ since the part of the path $\curve_0$ between $u_0$ and $y_i$ does not intersect $L_i$ since it is transverse to $\wcF_1$ in this part. 
Suppose that the part of the path $\curve_0$ 
from $x_{i+1}$ to $v_0$ intersects 
$L_{i+1}$ in a point $b$ not in $x_{i+1}$. This point $b$ is
in a leaf $z$ of $\wcG_E$ with $z \in [c_1,c_2]$.
Then $z$ is in $[x_k,y_k]$ with
$i+1 < k \leq j$ $-$ again because $z$ cannot be in 
$[x_{i+1},y_{i+1}]$ by transversality of the path there.
In particular $(x_{i+1},z)$ is in $\cP$ and
their interval difference is $k - (i+1) < j - i$, contradiction
of minimality of $(u_0,v_0)$. 
It follows that the leaf $L_v$ of $\wcF_1$ containing $v_0$ lies in the closure of $L_{i+1}^+.$ But $L_u$ and $L_v$ were equal by hypothesis: contradiction.
\end{proof}

Let $i$ be given by the claim, and let $L'$ be the 
leaf of $\wcF_1$ containing $ y_i$ and $x_{i+1}$.
Since $y_i$ is in $[c_1,c_2]$ it follows that $\curve_0$
intersects $y_i$ only once. 
Given $t>0$ small, let $p_t$ be the unique point in $\curve_t \cap L'$
which is very close to $y_i$. The point $p_0 \in \curve_0$ is a point in $y_i$. Similarly, one can define $p'_t$ the unique point in  $\curve_t \cap L'$
that is very close to $x_{i+1}$, and which converges to a point $p'_0 \in \curve_0$ in $x_{i+1}$. 

Recall that the disc $D_t$ is contained inside the leaf $E_t$, hence transverse to $L'.$ The intersection $L' \cap D_t$ may not be connected,
but it has a unique compact component $\ell_t$ which is an arc with
an endpoint in $p_t$ and another endpoint $w_t$ in the boundary of $D_t.$ Both $p_t$ and $w_t$
are in $\curve_t$ because they cannot intersect $I_t \subset L \neq L_t$.
We do not know that $w_t$ is near $p'_t$, but notice that $w_t$ intersects $\curve_t$  near one of the finitely many points of $L' \cap \curve_0$. 

So there is a subsequence $t_n \to 0$ with $w_{t_n}$ all converging
to a point $v$ in $\curve_0 \cap L'$. Let $c_v$ be the leaf of $\wcG$ through $v$. Note that we have $v \neq p_0$, and 
since $y_i$ intersects $\curve_0$ in only $p_0$, we have $c_v \neq y_i.$ 
The $\ell_{t_n}$ are converging to $y_i$ and $c_v$.   It follows that the length
of $\ell_{t_n}$ converges to infinity, and that the leaves $y_i, c_v$ in $\wcG_{L'}$ are non separated
in $\wcG_{L'}$. 
Notice that $c_v$ is also a leaf in $[c_1,c_2]$,
because $v \in \curve_0$.

We can restart the analysis with $L'$, $y_i, c_v$, which are non 
separated from each other in $\wcG_{L'}$ and $y_i, v$ are
contained in $E$.

If $c_v = x_{i+1}$, then $y_i$ and $x_{i+1}$ are non separated
in both $\wcG_E$ and in $\wcG_{L'}$. In this case we are done.

Suppose that $c_v$ is not $x_{i+1}$.
We consider the pseudo-interval $[y_i, c_v]$ in the leaf space of $\wcG_E.$ It is contained in $[c_1, c_2]$. Its interval length could still be $n$, but this
would force $i = 1$, and $c_v$ would have to lie in the last interval $[x_n, y_n].$  
Since $i = 1$, and $y_1, x_2$ are in the same leaf $L'$ of $\wcF_1$,
then  $x_2$ and $c_v$ lie in the same leaf $L'$ of $\wcF_1$.
Now we apply Claim \ref{i+1} to the pseudo-interval $[x_2, c_v]$ (which
is contained in $[c_1,c_2]$ and has length $< n$). The claim
shows that there is an integer $j>1$ such that $y_j$ and $x_{j+1}$ 
that are in the same leaf of $\wcF_1$. 
The difference is that we have started with a pseudo-interval 
$[x_2,c_v]$ which now has length $< n$.

Iterating this process each time either produces the conclusion
of Theorem \ref{nonsep} or decreases the length of the pseudo-interval
in question. Therefore this procedure has to eventually 
terminate and completes the proof of Theorem \ref{nonsep}.
\end{proof}

\subsection{Expansion in Anosov foliations}

In this section, we will provide a general result which is a consequence of the existence of Margulis measures for Anosov flows that we believe can be useful in general. Recall that a foliation $\cF$ is said to be an \emph{Anosov foliation}, if it is homeomorphic to the weak stable foliation of an Anosov flow $\Phi$. Note that by structural stability, one can assume that the Anosov flow is smooth (in particular $C^2$ in order to be able to apply the results in \cite{Ha}). Note that non-separated leaves produce some expansion in holonomy in one side of the foliation, here, we show that when the foliation is Anosov, this forces the expansion to hold on both sides. 

To state the result, let us give a very general definition. Recall that if $\cF$ is a transversally orientable codimension one foliation of $M$ we can construct a continuous family of local transversals $\{\tau_x:(-\eps,\eps) \to M\}_{x\in M}$ at every point in $M$ that we will keep fixed (these can be constructed using foliation boxes and partitions of unity).  Given a continuous curve $h: [0,T] \to M$ whose image is contained in a leaf $L$ of $\cF$ we can define the holonomy over $h$ to be the map defined on a neighborhood $(-\delta_1,\delta_2)$ (which depends on $h$ up to homotopy with fixed endpoints) to $(-\eps,\eps)$ that denotes the holonomy of $\cF$ from $\tau_{h(0)}((-\delta,\delta))$ to $\tau_{h(T)}((-\eps,\eps))$ over the curve $h$ (see \cite{Ca-Co} for this standard definition, note that here we consider the holonomy from the domains of the transversals rather than from the transversals). Note that the values of $\delta_1$ and $\delta_2$ are given as the maximal values on which the holonomy is defined from $\tau_{h(0)}((-\delta_1, \delta_2))$ to $\tau_{h(T)}((-\eps,\eps))$ by curves homotopic with fixed endpoints to $h$. 

For any continuous curve $r: [0,\infty) \to M$ contained in a leaf $L$ of $\cF$ one can define functions $\delta_1: [0,\infty) \to (0,\eps]$ and $\delta_2: [0,\infty) \to (0,\eps]$ so that the holonomy from $r(0)$ to $r(T)$ over $r$ is defined in the interval $(-\delta_1(T), \delta_2(T))$ (using the usual definition of holonomy by covering $r([0,T])$ by foliation charts and lifting the curve to nearby leaves). We say that $r: [0,\infty) \to M$  has \emph{expanding holonomy on one side} if there exists $\eps>0$ so that either $\delta_1$ or $\delta_2$ tend to $0$ as $T \to \infty$ (note here that the functions $\delta_1$ and $\delta_2$ may not be decreasing, so this says that given $\delta>0$ if $T$ is large enough, the holonomy maps an interval of `length' $\delta$ from one side to an interval of length larger than $\eps$). It has \emph{expanding holonomy} if this holds for both $\delta_1$ and $\delta_2$. Note that these are topological properties.

\begin{proposition}\label{prop-expAnosov}
Let $\cF$ be a Anosov foliation and let $r: [0, \infty) \to M$ a properly embedded ray contained in a leaf $L$ of $\cF$ with the property that it has expanding holonomy in one side. Then, it has expanding holonomy on both sides. 
\end{proposition}

\begin{proof} 
The conclusion is unchanged by taking any finite cover, so we 
may assume that the foliation $\cF$ is transversely orientable.

Up to homeomorphism, we can assume that $\cF$ is the weak stable foliation of a smooth Anosov flow $\Phi$. Note that the statement is invariant under homeomorphisms, so it is enough to work in this setting. 

The statement is independent of the choice of transversals and metric so, we choose to use the local unstable manifolds at each point, and we use the Margulis measure (see \cite{Ha}, this is the reason we assumed the flow was $C^2$) to put a metric on these transversals. It follows that for this metric, the holonomy along flowlines is exactly multiplying by $e^{ht}$ for a flowline of length $t$ where $h$ is the entropy of the flow $\Phi$. See in particular \cite[Lemma 3]{Ha}. 

Note also that inside a given weak stable leaf, we can also put a metric on strong stable leaves, and by continuity it follows that the distorsion of holonomy along strong stable leaves of length $\leq 1$ is bounded by a uniform constant. 

Given a ray $r: [0,\infty) \to M$ lying in a leaf $L$ as in the statement, we claim that the domain of definition of the holonomy in each side differs at most by some uniform constant. For this, note that for every $T>0$ there are times $t_1,t_2 >0$ so that $\Phi_{t_1}(r(0))$ and $\Phi_{t_2}(r(T))$ belong to a strong stable manifold of length $\leq 1$. Thus, $r|_{[0,T]}$ is homotopic with fixed endpoints to a concatenation of arcs $\alpha_1= \Phi_{[0,t_1]}(r(0))$ and arc $\alpha_2$ in the strong stable manifold of length $\leq 1$ and the arc $\alpha_3$ which is the arc $\Phi_{[0,t_2]}(r(T))$ oriented backwards. It follows from the considerations above, that up to a bounded constant, the holonomy from $r(0)$ to $r(T)$ is multiplying by $e^{h(t_1-t_2)}$ on both sides. In particular, if it expands in one side (which in particular means that $t_1-t_2 \to +\infty$ as $T \to \infty$) this implies that it must expand on both sides as claimed.  
\end{proof}

\section{Intersecting invariant graphs produces Reeb annuli}
\label{sec.intersect}

This section completes the proof of Theorem \ref{thm.three},
under the orientability assumptions on $\cG$ and
transverse orientability of $\cF_1, \cF_2$.
We then assume that $\cF_i$ are skew Anosov 
foliations which are uniformly equivalent, thus, we can assume that $\cF_1$ equal to the weak  stable foliation $\cF^s$ of an Anosov flow $\Phi$, and $\cF_2$ also a skew Anosov foliation, which
is isotopic (in the sense of  \S~\ref{sub.homotopic}) to the unstable  foliation $\cF^u$ of  $\Phi$.

Since $\cF_1, \cF_2$ are isotopic to $\cF^s, \cF^u$ respectively,
and $\Phi$ is a skew Anosov flow, it follows that a leaf
of $\wcF_1$  does not  intersect every leaf of $\wcF_2$.
In particular $\alpha$ and  $\beta$ (cf. Definition \ref{def.alphabeta}) are finite functions.

The purpose of this section is to prove the following
result which thanks to Proposition \ref{prop-nointersectionAnosov} completes the proof of Theorem \ref{thm.three}. 

\begin{theorem} \label{nointersect}
Suppose that the image of $\cD$ in the phase space $\cL_1 \times \cL_2$ intersects a $\pi_1(M)$ invariant
monotone graph. Then $\cF_1$ and $\cF_2$ 
both contain a Reeb annulus of $\cG.$  
\end{theorem}

We argue by a way of contradiction: we assume that the image of $\cD$ intersects a $\pi_1(M)$ invariant
monotone graph $\eta$, and that there are no Reeb annuli in $\cG$ restricted
to leaves of one of the foliations, let us say, $\cF_1$,
and we eventually arrive at  a contradiction.

Since the set of  invariant monotone graphs is discrete
and  there  are at least two of them (the graphs of
$\alpha$ and $\beta$),  then there is
a  $\mathbb{Z}$ worth of  them, which will be denoted
by $\{ \eta_k, k  \in  \mathbb{Z} \}$.  We choose them
nested with $k$, with  increasing $\cL_2$ coordinates with
$k$ increasing.
The fact that there are infinitely many follows  from
the  fact  that  $\cH$ is a free abelian group by Lemma \ref{graphgroup}. 

The following lemma describes fixed points of deck translations  in an  invariant monotone graph:

\begin{lemma} \label{lemma.same}
Let $\gamma$ in $\pi_1(M) \setminus \{ \mathrm{id}\}$
which fixes  a point $(x,y)$ in the $\pi_1(M)$  invariant
 monotone  graph $\eta$. Then $x, y$ are either
both contracting fixed points or  both repelling fixed points
for the action of $\gamma$ on $\cL_1$ and $\cL_2$
respectively.
\end{lemma}

Equivalently, this lemma states that the fixed points in the invariant graphs are either sinks or sources, while the fixed points outside the graphs are known to be saddle fixed points. 

\begin{proof}
Suppose that  $x$ is an attracting fixed point of $\gamma$
acting on $\cL_1$. 
Then $y = \beta(x)$ is 
 an attracting point for the  action of $\gamma$ on
$\cL_2$, since $\beta$ is a conjugacy, this completes the proof. 
\end{proof}

We will assume that: 

\begin{enumerate}
\item The image of $\cD$ intersects the graph $\eta_0$. 
\item There are no Reeb annuli of $\cG$ inside $\cF_1$. 
\end{enumerate}

Let $V$ be a torus obtained from $\cD^{-1}(\eta_0)$ as
in Proposition \ref{prop.closed}. The same proposition shows
that $\cG$ has a closed leaf in $V$. Consider a closed leaf $\ell$ of
$\cG$ in $V$ and associated deck transformation $\gamma_0$.
Then $\gamma_0$ has a fixed point $(x_0,y_0)$ in $\eta_0$.
Up to taking the inverse of $\gamma_0$ assume that 
$x_0$ is an attracting fixed point of $\gamma_0$ acting on
$\cL_1$.
By the previous lemma $y_0$ is also an attracting fixed point
of $\gamma_0$ acting on $\cL_2$.
In particular $\ell$ is either a repelling or an attracting orbit
for $\cG$.

Let $\tau_i$ be the ``one  step up" homeomorphisms
from  $\cL_i$ to themselves which are described in
subsection \ref{sub.defanosov}. Notice that each $\tau_i$ commutes
with any $\gamma$ in  $\pi_1(M)$.
According to Lemma \ref{le.fixedpoints}, the fixed points of $\gamma_0$ in $\cL_1$ form an increasing sequence $(x_i)_{i \in \zzzz}$ so that $x_i$ is attracting if $i$ is even, and repelling if $i$ is odd, and such that $x_{i+2}=\tau_1(x_i).$ Here $x_0$ is the original fixed point of $\gamma_0$ obtained
in the previous paragraph. We index similarly the fixed points of $\gamma_0$ in $\cL_2$ as a sequence $(y_j)_{j \in \zzzz}$.

We now describe all invariant graphs.
Let $f: \cL_1 \to \cL_2$ whose graph  is  $\eta_0$.

\begin{lemma} \label{lem.invariantgraphs}
The set of all invariant graphs is the collection:
 $\cC:=\{ \eta_k, k \in {\mathbb{Z}} \}$,
where $\eta_k = \{ (x, \tau^k_2(y)), \ {\rm with} \ 
(x,y) \in \eta_0 \}$.
\end{lemma}

\begin{proof}
Notice that  
$f \circ  \gamma = \gamma \circ f$ for any $\gamma$ in $\pi_1(M)$.
In  addition for any $\gamma$ in $\pi_1(M)$ then
$\gamma \circ  \tau_2 = \tau_2 \circ   \gamma$.
Therefore $\gamma \circ  \tau^k_2 \circ f =
\tau_2^k  \circ \gamma \circ f = \tau^k_2 \circ f \circ  \gamma$.
This implies that

$$\gamma(x,\tau^k_2 \circ f(x)) \ = \ (\gamma(x), \gamma \circ 
\tau^k_2 \circ f(x)) \ = \ (\gamma(x), \tau^k_2 \circ  f(\gamma(x))),$$

\noindent 
so it  is  on  the  graph  of $\tau^k_2 \circ f$,
and hence $\eta_k$ is $\pi_1(M)$ invariant. 
Also $\eta_k$ is  obviously a monotone graph so $\eta_k$ is
a  $\pi_1(M)$ invariant monotone  graph.

Finally, to  see that these are all  the invariant monotone   graphs,
suppose  there is an invariant monotone graph $\eta$ between
$\eta_0$ and $\eta_1$. Let $x$ in $\cL_1$ fixed by a non
trivial element $\gamma$ of $\pi_1(M)$, and so that $x$
is an  attracting fixed point of $\gamma$ acting on $\cL_1$. 
Then there is a fixed point $y$ of $\gamma$ acting on $\cL_2$
so that $y$ is in $(f(x),\tau_2(f(x)))$.   Notice that  $y$  is
repelling for $\gamma$ acting  on $\cL_2$. Now $(x,y)$ is in $\eta$ and $x$
is attracting for  $\gamma$ and $y$ is repelling. This 
contradicts the previous lemma.
This suffices to finish the proof.
\end{proof}

We have a converse to Lemma \ref{lemma.same}.

\begin{lemma} \label{lemma.saddle}
Let $\gamma$ be a non trivial deck transformation with a fixed
point  $(x,y)$ in  $\cL_1 \times \cL_2$ so that
$(x,y)$ is not in an invariant  monotone graph.
Then $(x,y)$ is a fixed point of saddle type of $\gamma$.
\end{lemma}

\begin{proof}
Let $(y_i)_{i \in  {\zzzz}}$ be an increasing sequence, which is
the collection of fixed points
of $\gamma$ in $\cL_2$. Using the description of the
previous lemma, we index the $y_i$ so that
$(x,y_{2i})$ is in $\eta_i$ for all $i$. 
Up to taking an inverse of $\gamma$  if 
necessary assume that $y$ is a repelling fixed point of 
$\gamma$ acting on $\cL_2$.
Let $i$ so that $(x,y)$ is between $\eta_i$ and $\eta_{i+1}$.
Again, by the description of the previous lemma, the
only fixed point of $\gamma$ in $(y_{2i}, y_{2i+2})$
is $y$. Since $y$ is repelling then $y_{2i}, y_{2i+2}$
are attracting for the action of $\gamma$ on $\cL_2$.
Notice that $(x,y_{2i})$ is in $\eta_i$, so Lemma \ref{lemma.same}
implies that $x$ is also
attracting for $\gamma$ acting on $\cL_1$. But 
$y$ is repelling for  $\gamma$ acting on $\cL_2$.
This proves the lemma.
\end{proof}

Denote by $\Omega = \cD(\mt)$,  the image of $\cD$. We now prove the following result which complements Proposition \ref{prop-hsdff}:

\begin{proposition} \label{prop.closedorbits}
Suppose that there are no Reeb annuli of $\cG$ in $\cF_1$.
Then the leaf space of $\wcG$ is Hausdorff and it is equivariantly
homeomorphic to $\Omega$.
\end{proposition}

The same result holds for $\cF_2$ in place of $\cF_1$. The proof is somewhat long and requires some preparation and a key lemma (Lemma \ref{lemma.closed} below). 

We assume that the leaf space of $\wcG$ is not Hausdorff
and produce a Reeb annulus of $\cG$ in a leaf of $\cF_1$. 

Suppose that $g_1, g_2$ are leaves of $\wcG$ which are not
separated from each other. Then there are $h_n$  leaves  of
$\wcG$ converging to $g_1 \cup g_2$.  Since $\cF_1$ is
$\rrrr$-covered it  follows that $g_1, g_2$ belong to the
same leaf $L_0$ of $\wcF_1$ (and similarly they belong 
to the same leaf of $\wcF_2$).
It follows that $g_1, g_2$ cannot be connected by a transversal
to $\wcG$ in $L_0$. Therefore there are $g'_1, g'_2$ leaves of $\wcG$
in $L_0$ which are non  separated from each other in the leaf
space of $\cG_{L_0}$. Since $\cF_2$ is $\rrrr$-covered,
then $g'_1, g'_2$ are in the same leaf of $\wcF_2$.
Theorem \ref{nonsep} implies that there are $f_1, f_2$ leaves of $\wcG$
which are in the same leaf $L$ of $\wcF_1$, and in the same leaf
$E$ of $\wcF_2$,  and so that $f_1, f_2$ are non  separated
from each other in $\wcG_L$ and non separated from
each other in $\wcG_E$.


The main step is proving the following:

\begin{lemma} \label{lemma.closed}
$\pi(f_i)$ is a closed leaf of $\cG$.
\end{lemma}

\begin{proof}
Let $r_i$ be the rays of $f_i$ such that $r_1$ and $r_2$ are not separated. Let $o_i = \pi(f_i)$ and  $\hat o_i = \pi (r_i)$. 

Assume by way of contradiction that $o_1$ is not
closed. Let $\cC$ be the collection of invariant graphs as in Lemma \ref{lem.invariantgraphs}. 

\begin{claim} \label{claim.limit}
The ray 
$\hat o_1$ can only limit in the tori obtained from
$\pi(\cD^{-1}(\cC))$.
\end{claim}

\begin{proof} 
Proposition \ref{prop.closed} 
shows that $\pi(\cD^{-1}(\cC))$ is a disjoint, finite
union of tori, which are foliated by $\cG$. 

Since $f_1$ is non separated from $f_2$, it follows that holonomy
of $\wcF_2$ is expanding along $f_1$ in the direction of the
non separated ray from $f_2$ and on the side that contains
$f_2$. This means the following: there is $\eps_0 > 0$ fixed,
so that given $q$ in $f_1$, and any leaf $\ell$ of $\wcG$ in
$L$ passing near $q$ and in the component of $L - f_1$ containing
$f_2$, then holonomy of $\wcG$ starting at $q$ and following $\ell$
eventually sends
$\ell$ farther than $\eps_0$ from $f_1$ and it never returns $\eps_0$ 
close to $f_1$. This is because $\ell$ eventually has to get close
to $f_2$.
The same happens for holonomy of $\cF_1$ since $f_1, f_2$
are also non separated in their $\wcF_2$ leaves.

This means that holonomy of $\wcF_2$ along $f_1$ is expanding
on one side of $f_1$. Using Proposition \ref{prop-expAnosov} we deduce that the holonomy is expanding on both sides. 

Let now $p$ be an accumulation point of the ray of $\hat o_1$.
Given two returns which are $\ll \eps$ from each other,
the fact that the holonomy of $\wcF_1$ and $\wcF_2$ along
$f_1$ is expansive, implies that the return map of a small
disk transverse to $\cG$ sends it to a strictly bigger 
transverse disk. In other words, there is a periodic
orbit of $\cG$ crossing this disk.

Assume now that $\hat o_1$ does not limit only on $\pi(\cD^{-1}(\cC))$.
Then since this is a finite union of disjoint tori,
it follows that $\hat o_1$ limits on a point that is $> a_0 > 0$ away from
this set. Using the above procedure we can find a periodic orbit
of $\cG$ which is not in $\pi(\cD^{-1}(\cC))$. But this periodic
orbit is expanding, so the deck transformation associated to it,
has a corresponding fixed point in $\cL_1 \times \cL_2$ which
is contracting. This contradicts Lemma \ref{lemma.saddle}.

This proves the claim.
\end{proof}

We now know that $\hat o_1$ can only 
limit on $\pi(\cD^{-1}(\cC))$. In particular there is a fixed
surface $S$ in $\pi(\cD^{-1}(\cC))$ so that 
$\hat o_1$ can only limit on $S$. 

Recall the dynamics of $\cG$ on $S$ which
was described in Proposition \ref{prop.closed} and Lemma \ref{lemma.same}: there are finitely many
closed leaves of $\cG$ on $S$ and each one is either attracting
or repelling. In addition each closed leaf is attracting or
repelling not just restricted to $S$, but in $M$.

Since $\hat o_1$ only limits on $S$, the above facts imply that
$\hat o_1$ only limits on a closed curve of $\cG$ in $S$.
In other words $\hat o_1$ is asymptotic to this curve. 
Suppose first that this closed curve is attracting for $\cG$.
Then $\hat o_1$ is in the attracting set for this curve,
and $f_1$ cannot be expanding in either $L$ or $E$.
This contradiction shows that the closed curve has
to be repelling. But then $\hat o_1$ is asymptotic to this curve and this
can only happen if $o_1$ is this closed repelling curve.  

This finishes the proof of Lemma \ref{lemma.closed}.
\end{proof}

\begin{proof}[Proof of Proposition \ref{prop.closedorbits}]
The lemma shows that $\pi(f_1), \pi(f_2)$ are closed
curves, and they are both in $\pi(F)$. 
Since $\cF_1$ and $\cF_2$ are Anosov foliations, this implies that $\pi(F)$ is an annulus.
Since $\pi(f_2)$ is also closed, then $\pi(f_1), \pi(f_2)$
are isotopic in $\pi(F)$.
Let $\delta$ be
the deck transformation associated to $\pi(f_1)$, this
is  a generator of $\pi_1(\pi(F))$. In particular
$\pi(f_1), \pi(f_2)$ are distinct (and hence disjoint)
curves in $\pi(F)$. 
If there is  a transversal to $\cG$ in $\pi(F)$ from
$\pi(f_1)$ to $\pi(f_2)$, it now follows that 
there is a transversal from $f_1$ to $f_2$, a contradiction.
Therefore there is no transversal from $\pi(f_1)$ to $\pi(f_2)$
in $\pi(F)$ and since $\pi(F)$ is an annulus, it now
follows that there is a Reeb annulus of $\cG$ in $\pi(F)$.

This finishes the proof that if the leaf space of $\wcG$ is not
Hausdorff, then there is a Reeb annulus of $\cG$ in a leaf
of $\cF_1$.

From this, it is easy to get the other statement of the Proposition:
suppose that $\cD$ does not induce an injective map 
from the leaf space of
$\wcG$ to $\Omega$. Therefore there are two leaves
$f_1, f_2$ of $\wcG$ which map to the same point in $\Omega$,
and hence to the same point in $\cL_1 \times \cL_2$.
In particular $f_1, f_2$ are in the same leaf $L$ of $\wcF_1$.
If there is a transversal to $\wcG$ in $L$ from $f_1$ to
$f_2$, then there is a transversal to $\wcF_2$
from $f_1$ to $f_2$. In particular $f_1, f_2$ cannot
be in the same $\wcF_2$ leaf. This contradicts the fact
that $f_1, f_2$ map to the same point in $\cL_1 \times \cL_2$.
This finishes the proof. 
\end{proof}

Recall that $f:  \cL_1 \to \cL_2$ is the homeomorphism
whose graph is  $\eta_0$.
We follow the setup as in  \S~\ref{sec.nointersect} as  follows:
we use that $\cF_1$ is
the weak stable foliation of an Anosov flow  $\Phi$.
Let $\cL^s, \cL^u$ be the leaf spaces of the weak stable and weak unstable foliations
of $\Phi$. Let $v_1: \cL_1 \to  \cL^s$  be the map given by identifying $\cF_1$ with $\cF_s$.
Define $v_2: \cL_2 \to \cL^u$ by $v_2(E) = \alpha_0 \circ
\nu_1 \circ f^{-1}(E)$, where $\alpha_0$ is the lower
bound map associated  with $\Phi$. Finally let 
$\nu(x,y) = (\nu_1(x), \nu_2(y))$. As in \S~\ref{sec.nointersect} this map is $\pi_1(M)$ equivariant. 
If we let $\Omega_i$ be the region in $\cL_1 \times \cL_2$
between  $\eta_i$ and  $\eta_{i+1}$ then the action of
$\pi_1(M)$ on $\Omega_i$ is  equivariantly homeomorphic
to the action  of  $\pi_1(M)$ on the orbit space of $\Phi$.
We remark that in  this analysis it is not relevant whether
$\alpha$ is finite or not. All that is used is that there
is an  invariant  graph.

\begin{proof}[ Proof of Theorem \ref{nointersect}]
Now we will arrive at a contradiction, which stems from
the assumption that $\cF_1$ has no Reeb annuli of $\cG$ and
this will finish the proof.

Proposition \ref{prop.closedorbits} implies  that the leaf
space of $\wcG$ is equivariantly homeomorphic to $\Omega = \cD(\mt)$,
and each $\cD^{-1}(p)$ is a  single leaf of $\wcG$ for any 
$p$ in $\Omega$.  
By assumption $\Omega$ intersects an invariant monotone  graph,
let this be $\eta_i$, hence $\Omega$  contains $\Omega_i$.   
Let $p$ in $\Omega$ so  that  $\ell = \pi(\cD^{-1}(p))$ is  a closed
orbit of $\cG$.  By Lemma \ref{lemma.same}  we can choose
$\ell$ so that it is  an  
attracting 
orbit of $\cG$. 
Here $\Omega_i$ is equivariantly homeomorphic to the 
orbit space of $\Phi$, and the periodic orbits of  $\Phi$
are dense in $M$. Hence there are $p_k$ in  $\Omega_i$,
$p_k  \to p$, which are fixed by non  trivial elements
$\gamma_k$ of $\pi_1(M)$. Then $\ell_k = \pi(\cD^{-1}(p_k))$
are periodic orbits which are converging to $\ell$. 
But this is impossible since $\ell$ is an attracting
periodic orbit.

This is a contradiction, and finishes the proof of Theorem
\ref{nointersect}.
\end{proof}

\section{The non orientable cases} \label{s.lastone}

\subsection{Theorem \ref{thm.three}}

First we prove Theorem \ref{thm.three} without
the orientation conditions.

\begin{proof}[Proof of Theorem \ref{thm.three}.]
First we lift to a finite cover $N$ of $M$ so that
the lifts $\cF^N_1, \cF^N_2$ of $\cF_1, \cF_2$ are transversely orientable,
and the lift $\cG^N$ of $\cG$ is orientable.
If $\cF^N_1$ or $\cF^N_2$ contain a Reeb annulus
of $\cG^N$, then this annulus projects to a Reeb surface
of $\cG$ in $\cF_1$ or $\cF_2$. In this case we are finished.

Otherwise $\cG^N$ is homeomorphic to the flow foliation
of an Anosov flow $\Phi^N$. The stable and unstable foliations of
$\Phi^N$ are (up to switching) $\cF^N_1, \cF^N_2$, and they
project to $\cF_1, \cF_2$ in $M$. 

Now consider a double cover $M_2$ of $M$ so that $\cG$ lifts
to an orientable foliation $\cG^2$, and $\cF_1, \cF_2$
lift to $\cF^2_1, \cF^2_2$ respectively. 
This cover is normal so there is a deck transformation
$\zeta$ of $M_2$ so that $M = M_2/_{\langle \zeta \rangle}$.
In addition $\zeta$ preserves leaves of $\cF^2_i$.
Orient $\cG^2$ so that in a leaf of $\cF^2_1$ the orbits
of $\cG^2$ are forward asymptotic. 
Now $\zeta$ sends a leaf of $\cF^2_1$ to another leaf of $\cF^2_1$
and sends the leaves of $\cG^2$ to leaves of $\cG^2$, in particular
it preserves forward asymptotic behavior, and therefore 
preserves the orientation of curves of $\cG^2$.

We conclude that $\cG$ is orientable. It is also tangent
to two foliations $\cF_1, \cF_2$ and orbits of $\cG$ are
forward asymptotic in leaves of $\cF_1$ and backwards
asymptotic in leaves of $\cF_2$. It follows that
$\cG$ is the flow foliation of a topological Anosov flow.
By work of Shannon \cite{Sh} this topological Anosov
flow is orbitally equivalent to an Anosov flow.

This finishes the proof of the Theorem \ref{thm.three}.
\end{proof}

\subsection{Product Anosov foliations again}

Here we generalize the result of \S~\ref{s.prod} to the
non orientable case and use the same notation as in that section. Let $(\cF_1, \cF_2)$ be a pair of minimal transverse foliations on $M(A)$. 

Lift to a finite normal cover where the foliations
are transversely orientable. This finite normal cover
is homeomorphic to $M(A')$ (on which $A'$ is just a power of $A$). Let $\Phi'$ be the suspension
Anosov flow in $M(A')$. In particular it is the lift
of the suspension flow $\Phi$ in $M(A)$. Let $\cF'_i$ be the respective
lifts to $M(A')$. 
Proposition \ref{productproduct} has 3 options.
In option (1), $\cF'_1$ is isotopic to say the stable
foliation of $\Phi'$ and $\cF'_2$ is isotopic to the unstable
foliation of $\Phi'$. It follows that $\cF_1$ is isotopic
to the stable foliation of $\Phi$, and $\cF_2$ is isotopic
to the unstable foliation of $\Phi$. It follows that $\cG$
is isotopic to the orbit foliation of $\Phi$.

In case (2), both $\cF'_1, \cF'_2$ are isotopic to the stable
foliation of $\Phi'$. So both $\cF_1, \cF_2$ are isotopic to the
stable foliation of $\Phi$. This implies that $\cG$ is isotopic to
the strong stable foliation $\cF^{ss}_A$. Similarly
one deals with case (3).

\section{Analysis in Seifert manifolds}\label{s.circlebundle}

Here  we prove Theorem \ref{thm.four}
which implies Theorem \ref{thm.two}. We first restate Theorem \ref{thm.four}:

\begin{theorem}
Suppose  that  $\cF_1, \cF_2$ are transverse foliations in
a Seifert manifold. 
Suppose that $\cF_1, \cF_2$ are Anosov  foliations
and that $\cF_1$ does not contain Reeb surfaces
of  the intersection foliation $\cG$. Then $\cG$ is 
topologically equivalent to the flow foliation 
of an Anosov flow.
\end{theorem}

Theorem \ref{thm.two} follows because, as remarked
before,  in $T^1 S$ a minimal foliation is an Anosov foliation.

First we observe that according to \cite{Ba2}, any Anosov flow on a Seifert manifold, lifted in some finite cover $M'$, is the pull-back by a finite covering of the geodesic flow on the unit tangent bundle of some surface. Hence, lifting everything in $M'$, the intersection foliation $\cG'$ is the lift of $\cG$. The lifted foliation $\cF'_1$ does not contain Reeb surfaces of $\cG'$ since we assume that $\cF_1$ does not contain Reeb surfaces. And if $\cG'$ is topologically equivalent to the flow foliation of an Anosov flow, the same holds for $\cG.$

Therefore, up to a finite cover, it suffices to prove Theorem \ref{thm.four} in the case where $M$ is a circle bundle over a closed oriented hyperbolic surface $\Sigma$, and where $\cF_1$, $\cF_2$ are transversely orientable.





\subsection{Preliminaries} 
For more details on this part, we refer to  \cite{Barb-Fe2}.
Suppose that $\cF_1$  is
the weak stable foliation of an  Anosov flow $\Phi_1$ and
$\cF_2$ is the weak unstable foliation of an Anosov flow $\Phi_2$.
Let  $\Sigma$ be the base surface  of a fibration by circles
in $M$. Notice that the fibration by circles in $M$ is unique
up to  isotopy.  We fix one such Seifert fibration $p: M \to \Sigma$.
Fix  a hyperbolic metric in $\Sigma$ and let $M_1(\Sigma)$ be
the unit tangent bundle of $\Sigma$, with projection 
$p^0: M_1(\Sigma) \to  \Sigma$.
Let $\Phi_0$ be the geodesic flow in $M_1(\Sigma)$.

Ghys \cite{Gh} proved that for $i=1,2$, the flow
 $\Phi_i$ is a  reparametrization of
the pullback of 
$\Phi_0$ by some finite cover $q_i: M \to M_1(\Sigma)$
which is along the fibers, that is, $p$ is isotopic to $p^0 \circ q_i.$

It means that the action of $\pi_1(M)$ on $\cL_i$ is topologically conjugate to the action on the universal covering $\widetilde{\mathrm{RP}}^1$ of the projective line 
given by a representation $\rho_i: \pi_1(M) \to \widetilde{\mathrm{PSL}}(2, \rrrr)$ which is faithful, discrete and with cocompact image.

There is an exact sequence:
$$1 \to \zzzz \to \widetilde{\mathrm{PSL}}(2, \rrrr) \to \mathrm{PSL}(2,\rrrr) \to 1$$
where the cyclic kernel $\zzzz$ is the group of deck transformations of the universal covering $\widetilde{\mathrm{PSL}}(2, \rrrr) \to \mathrm{PSL}(2,\rrrr)$.
We denote by $\delta$ a generator of this kernel. We orient $\widetilde{\mathrm{RP}}^1$ so that $\delta$ is an increasing homeomorphism of
$\widetilde{\mathrm{RP}}^1$.

The action of $\pi_1(M)$ on the phase space $\cL_1 \times \cL_2$ is topologically conjugate to the diagonal action on $\widetilde{\mathrm{RP}}^1 \times \widetilde{\mathrm{RP}}^1$ defined
by the pair $(\rho_1, \rho_2).$

Now consider the projections of these representations into $\mathrm{PSL}(2, \rrrr).$ They induce Fuchsian representations of $\pi_1(\Sigma)$ into $\mathrm{PSL}(2, \rrrr).$
Therefore, the induced actions on $\mathrm{RP}^1$ are topologically conjugated. It means that modifying the identification between $\cL_2$ and $\widetilde{RP}^1$ if necessary, we
can assume that there is a morphism $h: \pi_1(M) \to \zzzz$ such that, for every $\gamma$ in $\pi_1(M)$ we have:
$$\rho_2(\gamma) = \rho_1(\gamma)\delta^{h(\gamma)}.$$

Therefore, we have identified $\cL_1 \times \cL_2$ with $\widetilde{\mathrm{RP}}^1 \times \widetilde{\mathrm{RP}}^1$ so that the action of an element $\gamma$ on a element $(x,y)$ of $\widetilde{\mathrm{RP}}^1 \times \widetilde{\mathrm{RP}}^1$ is given by:
$$\gamma.(x,y) = (\rho_1(\gamma), \rho_1(\gamma)\delta^{h(\gamma)}).$$

%

\subsection{The invariant family of monotone graphs} 

Now consider the family of curves in $\widetilde{\mathrm{RP}}^1 \times \widetilde{\mathrm{RP}}^1$  given by

$$\eta_k \ := \ 
\{ \ (x,\delta^k(x)),  \ x \in\widetilde{\mathrm{RP}}^1 \ \}.$$

\noindent
These are the graphs of the powers $\delta^k.$ Every curve $\eta_k$ is not individually preserved by $\pi_1(M).$ But:

\begin{lemma}\label{le.h}
The family $\cC := \{ \eta_k, \ k \in {\bf Z} \}$ 
is discrete, and invariant under the action of $\pi_1(M)$.
\end{lemma}

\begin{proof}
The discrete property is obvious since $\delta$ is strictly increasing. 

Consider $\eta_k$ for some $k$. Let $\gamma$ be an element of $\pi_1(M)$.
For every $x$ in $\widetilde{\mathrm{RP}}^1$ we have:
\begin{eqnarray*}
\gamma.(x, \delta^k(x)) & = & (\rho_1(\gamma)(x), \rho_1(\gamma)\delta^{h(\gamma)}\delta^k(x)) \\
  & = &  (\rho_1(\gamma)(x), \delta^{h(\gamma) + k} \rho_1(\gamma)(x))
\end{eqnarray*}

therefore, $\gamma(\eta_k) = \eta_{k + h(\gamma)}$. The lemma follows.
\end{proof}

\subsection{Using the invariant collection}

We will now we  use the invariant collection $\{ \eta_k, k \in {\mathbb{Z}} \}$
of the previous 
subsection to obtain a proof of  
Theorem \ref{thm.four}. We hence assume that $M$ is a Seifert manifold, and that that  $\cF_1, \cF_2$ are transverse Anosov foliations
and that $\cF_1$ does not contain Reeb surfaces
of  the intersection foliation $\cG$. Since $M$ contains an Anosov foliation, some finite cover of $M$ is a circle bundle. Since a finite quotient of an Anosov is an Anosov flow, we can assume without loss 
of generality that $M$ is a circle bundle.

\vskip .1in
\noindent
{\bf  {Case 1}} $-$ Suppose that $\alpha$ (or $\beta$)  is finite.

It follows that the actions  of $\pi_1(M)$ on $\cL_1, \cL_2$
are  conjugate by a   homeomorphism that commutes  with
the action by $\pi_1(M)$. It follows that $\cF_1, \cF_2$
are isotopically topologically conjugate. The proof now follows from
Theorem \ref{thm.three}.

\vskip .1in
\noindent
{\bf {Case 2}} $-$ Suppose that  the image of   $\cD$ is 
all of $\cL_1 \times \cL_2$.

In particular it follows that $\cD(\mt)$ intersects
all the curves in $\cC$ and $\cD^{-1}(\cC)$ is non empty.

Since $\cC$ is a discrete, $\pi_1(M)$ invariant collection of
monotone graphs, we can use Proposition \ref{prop.closed}:
each component of $\cD^{-1}(\cC)$ 
projects to $M$ into a torus, which
is saturated by flow lines of $\cG$. Each such torus has
finitely many closed leaves of $\cG$, each closed leaf
is an attracting or repelling orbit of $\cG$.

We need the result of Lemma \ref{lemma.saddle} in our 
case:

\begin{lemma} \label{lemma.saddle2}
Let $\gamma$ be a non trivial element in $\pi_1(M)$ which
fixes a point $p$ in $\cL_1 \times \cL_2$, so that $p$ is
not in $\cC$. Then $p$ is a saddle fixed point
of $\gamma$ acting on $\cL_1 \times \cL_2$.
\end{lemma}

\begin{proof}
Since $p$ is not in $\cC$ there is a unique $k$ so that
$p$ is between $\eta_k$ and $\eta_{k+1}$.
Let $p = (x,y)$. Recall that $\eta_k$ is the
graph of $\delta^k$.
Since $\gamma$ fixes $p$ then it fixes both $\eta_k$ and $\eta_{k+1}$.
Therefore, according to the arguments in the proof 
of Lemma \ref{le.h} we have $h(\gamma)=0.$ Therefore, the action of
$\gamma$ on $\widetilde{\mathrm{RP}}^1 \times \widetilde{\mathrm{RP}}^1$ is the diagonal action of $\rho_1(\gamma).$ The lemma follows.
\end{proof}

The proof now proceeds as in section \S~\ref{sec.intersect}.
First we consider the analogue of Proposition \ref{prop.closedorbits}:
suppose that the leaf space of $\wcG$ is not Hausdorff, then 
we find $f_1, f_2$ leaves of $\wcG$ which are in the same leaf
$L$ of $\wcF_1$ and the same leaf $E$ of $\wcF_2$ and which
are non separated in $\wcG_L$ and $\wcG_E$. This just uses
Theorem \ref{nonsep}. Then we prove the analogue of
Lemma \ref{lemma.closed} that $\pi(f_i)$ is a closed leaf
of $\cG$. The key step is the analogue of Claim \ref{claim.limit}
which shows that an appropriate ray of $\pi(f_i))$ (the projection
of a ray of $f_i$ non separated from a ray of $f_j$ ($j \not = i$))
can only limit on $\pi(\cD^{-1}(\cC))$. We have established that
in our case also $\pi(\cD^{-1}(\cC))$ is a finite collection
of tori, saturated by leaves of $\cG$,
and that in each component there are finitely many closed
leaves of $\cG$, each of which is either attracting or
repelling. The rest of the proof is exactly as in the proof
of Claim \ref{claim.limit}. Once the claim is proved, then
the proof of the analogous statement to
Lemma \ref{lemma.closed} is exactly the same as in
Lemma \ref{lemma.closed}. Once that is done, the rest of the
proof is exactly as in the proof of Proposition \ref{prop.closedorbits}.

This shows that, as in section \S~\ref{sec.intersect},
the leaf space of
$\wcG$ is Hausdorff, and $\cD$ induces a homeomorphism from
the leaf space of $\wcG$ to $\cD(\mt)$. 
Using Lemma \ref{lemma.saddle2} we obtain 
a contradiction as in section \S~\ref{sec.intersect}.
The contradiction is obtained
by assuming that there are no Reeb annuli of $\cG$ in
leaves of $\cF_1$.
This contradiction shows that if there are no Reeb annuli,
then Case 2 cannot happen. 

This finishes the analysis. In particular this proves
Theorem \ref{thm.four} under the orientation conditions.



\end{document}